\documentclass[11pt,leqno]{amsart}
\usepackage{amssymb,graphicx,color} 
\usepackage{mathrsfs}

\begin{document}
\theoremstyle{plain}
\newtheorem{thm}{Theorem}[section]
\newtheorem{prop}[thm]{Proposition}
\newtheorem{lem}[thm]{Lemma}
\newtheorem{cor}[thm]{Corollary}
\newtheorem{deft}[thm]{Definition}
\newtheorem{hyp}{Assumption}
\newtheorem*{KSU}{Theorem (Kenig, Sj\"ostrand and Uhlmann)}

\theoremstyle{definition}
\newtheorem{rem}[thm]{Remark}
\numberwithin{equation}{section}
\newcommand{\eps}{\varepsilon}
\renewcommand{\d}{\partial}
\newcommand{\dd}{\mathrm{d}}
\newcommand{\e}{\mathrm{e}}
\newcommand{\re}{\mathop{\rm Re} }
\newcommand{\im}{\mathop{\rm Im}}
\newcommand{\ch}{\mathop{\rm ch}}
\newcommand{\R}{\mathbf{R}}
\newcommand{\C}{\mathbf{C}}
\renewcommand{\H}{\mathbf{H}} 
\newcommand{\N}{\mathbf{N}}
\newcommand{\D}{\mathcal{C}^{\infty}_0} 
\newcommand{\supp}{\mathop{\rm supp}}
\hyphenation{pa-ra-met-ri-zed}
\title[]{Stability estimates for a Magnetic Schrodinger operator with partial data}
\author[]{Leyter Potenciano-Machado }
\address{Departamento de Matem\'aticas, Universidad Aut\'onoma de Madrid, Campus de Cantoblanco, 28049 Madrid, Spain}
\email{leyter.potenciano@uam.es}
%\address{Departamento de Matem\'aticas, Universidad Aut\'onoma de Madrid, Campus de Cantoblanco, 28049 Madrid, Spain}
%\email{alberto.ruiz@uam.es}
\begin{abstract}

In this paper we study local stability estimates for a magnetic Schr\"odinger operator with partial data on an open bounded set in dimension $n\geq 3$. This is the corresponding stability estimates for the identifiability result obtained by Bukgheim and Uhlmann \cite{BU} in the presence of magnetic field and when the measurements for the Dirichlet-Neumann map are taken on a neighborhood of the illuminated region of the boundary for functions supported on a neighborhood of the shadow region. We obtain $\log\log$-estimates for  magnetic potential and  $\log\log\log$ for electrical potential.

\end{abstract}
\maketitle
\setcounter{tocdepth}{1} 
\tableofcontents

\section{Introduction}

Let $\Omega \subset \mathbb{R}^n $ ($n\geq 3$) be an open bounded set with $C^{\infty}$ boundary, denoted by $\partial \Omega$. We consider the following magnetic Schr\"odinger operator
\begin{equation}
\mathcal{L}_{A,q}(x, D) := \sum_{j=1}^{n}\left( D_j + A_j(x)  \right)^2 + q(x) = D^2 + A\cdot D + D\cdot A + A^2 +q,
\end{equation}
where $D=-i \nabla$, $A=\left(  A_j \right)_{j=1}^{n}\in C^2\left( \overline{\Omega}; \mathbb{R}^n \right)$ is a magnetic potential and $q\in L^{\infty}\left( \Omega \right) $ is an electrical potential. The inverse boundary value problem (IBVP) under consideration in this article is to recover information  (inside $\Omega$)  about the magnetic and electrical potentials from measurements on subsets of the boundary. Roughly speaking we divide the boundary $\partial\Omega$ in two open subsets, $F$ and $B$. In this setting and if $0$ is not an eigenvalue of $\mathcal{L}_{A,q}$, we define the partial DN map as follows:
%\begin{align*}
%\Lambda_{F}:& H^{\frac{1}{2}}(\partial \Omega)\rightarrow H^{-\frac{1}{2}}(\partial \Omega)\\
%& f \rightarrow (\partial_\nu + iA\cdot \nu)u|_{F},
%\end{align*}
%where $u$ is a solution of (\ref{artif}) and $\supp f \subset \partial\Omega$.
\begin{align*}
\Lambda_{A,q}^{B\rightarrow F}:& H_B^{\frac{1}{2}}(\partial \Omega)\rightarrow H^{-\frac{1}{2}}(\partial \Omega)\\
& f \rightarrow (\partial_\nu + iA\cdot \nu)u|_{F},
\end{align*}
 where $\nu$ is the exterior unit normal of $\partial\Omega$, the set $H_B^{\frac{1}{2}}(\partial \Omega)$ consists of all $f\in H^{\frac{1}{2}}(\partial \Omega)$ such that $\supp f \subset \overline{B}$ (we will call this condition \textquotedblleft support constraint \textquotedblright) and  $u\in H^{1}(\Omega)$ is the unique solution of the following Dirichlet problem:
 \begin{equation}
\label{artif}
     \begin{cases}
            \mathcal{L}_{A,q}u=0  & \text{ in } \Omega \\ u|_{\d \Omega} = f.
     \end{cases}
\end{equation}
In the cases that $F$ or $B$ are not equal to $\partial\Omega$, we say that that  the inverse boundary value problem   has partial data. According to the choice  of the sets $F$ and  $B$, we can distinguish  several types of  partial data results  that we briefly describe.\\

In the absence of a magnetic potential ($A\equiv 0$), the pioneering work, which   we describe  as illuminating $\Omega$  from infinity was obtained by Bukgheim and Uhlmann \cite{BU}. They consider  a direction $\xi\in S^{n-1}$   and  $F\subset \partial\Omega$  to be   a neighborhood of  the $\xi$-illuminated face or front region, defined  as
\begin{equation}\label{omm}
\partial \Omega_{-,0}(\xi)= \left \{    x\in \partial\Omega  \; : \; \left \langle \xi, \nu(x) \right \rangle < 0 \right \}.
\end{equation}
In their  work they considered $B=\partial\Omega$. They obtained the identifiability result: if $\Lambda_{0,q_1}= \Lambda_{0,q_2}$ then $q_1=q_2$. The corresponding stability estimates were derived by Heck and Wang \cite{HW}. Later, Kenig, Sjostrand and Uhlmann \cite{KSU} obtained  a similar result when  $F$ and $B$ are neighborhoods  respectively   of the illuminated  and shadow boundary regions of $\Omega$ from  a point $x_0$ (out of the convex hull of $\Omega$), which are defined  by
\begin{equation}\label{cset}
 \partial \Omega_{-,0}(x_0)= \left \{    x\in \partial\Omega  \; : \; \left \langle x-x_0, \nu(x) \right \rangle < 0 \right \}
\end{equation}
and 
\[
 \partial \Omega_{+,0}(x_0)= \left \{    x\in \partial\Omega  \; : \; \left \langle x-x_0, \nu(x) \right \rangle > 0 \right \}.
\]
Notice that in this case if the domain is  strictly convex then $F $ could be arbitrary small.\\
 
In the  case of illuminating from infinity the supporting  set $B$ could also be restricted to  a neighborhood of the  shadow  region  from infinity. In the case  of $A=0$ the corresponding stability estimates with the support constraint were derived by  Caro, Dos Santos Ferreira and Ruiz \cite{CDSFR}, using Radon transform and for illumination from a point without the support constraint in \cite{CDSFR1} by using the  geodesic ray transform on the sphere.  In both cases, they obtained log log-estimates.\\

On the other hand, as it was noted in \cite{Sun}, in the presence of a magnetic potential ($A\not \equiv 0$) there exists a gauge invariance of the DN map. To be specific, if $\varphi\in C^1(\overline{\Omega})$ is a real valued function with $\varphi |_{\partial\Omega}=0$, then $\Lambda_{A,q}= \Lambda_{A+\nabla \varphi, q}$. Hence for the identifiability problem we only  expect  to prove that $dA_1=dA_2$ and $q_1=q_2$. Here we consider the magnetical potential $A$ as a 1-form as follows
\[
A=\sum_{j=1}^{n}A_jdx_j, \quad A= (A_1, A_2, \ldots, A_n) 
\]
 and
 \[
 dA= \sum_{1\leq j<k\leq n}\left( \partial_{x_j}A_k-\partial_{x_k}A_j  \right) d_{x_j}\wedge d_{x_k}.
 \]

We mention the results concerning to full data, that is to say, when $F=B=\partial\Omega$. Sun proved identifiability under the assumption of the smallness of  the magnetic potential in a suitable space \cite{Sun}. In \cite{NSU} the smallness was removed for $C^2$ and compactly supported magnetic potential and $L^\infty$ electrical potential. Finally, these results were extended by Krupchyk and Uhlmann \cite{KU} for both, magnetic and electrical potentials in $L^{\infty}$.\\

The identifiability result in  the  case of illumination from  a point for $B=\partial \Omega$ and in the presence of  a magnetic potential is due  to Dos Santos Ferreira, Kenig, Sjostrand  and Uhlmann \cite{DSFKSjU}.  It was extended  by Chung \cite{Ch} to the  case where the support constraint is a neighborhood $B$ of the shadow boundary.\\

To the best of our knowledge, the only stability result  with partial data in the presence of a magnetic potential was obtained by Tzou \cite{Tz}. He considered complete data and also partial data from infinity  without the support constraint. He obtained $\log\log$-stability estimates.\\

The main goal of this article is to derive stability estimates for the case  of Bukhgeim and Uhlmann in the presence of a magnetic potential with the  additional  support constraint on $B$, a neighborhood of the shadow boundary from infinity.\\

We denote by $C_i$ ($i\in \mathbb{Z}^+$) a positive constants which might change from formula to formula. This constants should depend only on $n, \Omega$ and the priori bounds for magnetic and electrical potentials.\\

Before stating our results we introduce some notation following \cite{CDSFR}. Given a direction $\xi\in S^{n-1}$ and $\epsilon \geq 0$, we define the $(\xi, \epsilon)$-illuminated face of $\partial \Omega$ as
\[
\partial \Omega_{-, \epsilon}(\xi)= \left \{    x\in \partial\Omega  \; : \; \left \langle \xi, \nu(x) \right \rangle < \epsilon \right \},
\]
and the $(\xi, \epsilon)$-shadowed face as
\[
\partial \Omega_{+, \epsilon}(\xi)= \left \{    x\in \partial\Omega  \; : \; \left \langle \xi, \nu(x) \right \rangle > -\epsilon \right \},
\]
where $\nu(x)$ denotes the exterior unit normal vector at $x$. Let $N$ be an open subset of $S^{n-1}$ and define the sets
\begin{equation}\label{FNBN}
F_N= \bigcup_{\xi \in N}\partial\Omega_{-,0}(\xi), \quad B_N= \bigcup_{\xi \in N}\partial\Omega_{+,0}(\xi).
\end{equation}
Now let $F$ and $B$ be open neighborhoods on $\partial\Omega$ of $F_N$ and $B_N$, respectively; and  let $\chi$ be a cutoff function supported on $F$ such that it is equals to $1$ on $F_N$. Denote by $H^{1/2}_B(\partial \Omega)$  the set consisting of all the functions $f\in H^{1/2}(\partial\Omega)$ such that $\supp f\subset \overline{B}$.  We define the partial DN map $\Lambda_{A,q}^{\sharp}:H^{1/2}_B(\partial \Omega) \rightarrow H^{-1/2}(\partial \Omega)$, as follows:
\[
\Lambda_{A,q}^{\sharp} f = \chi \Lambda_{A,q} f.
\]
We consider the associated operator norm defined by
\begin{equation}\label{partialnorm}
\left \| \Lambda^{\sharp}_{A,q} \right \|_{H^{1/2}_B(\partial\Omega)\rightarrow H^{-1/2}(\partial\Omega)} = \underset{\underset{  \left \| f \right \|_{H^{1/2}(\partial\Omega)}=1}{f\in H^{1/2}_B(\partial\Omega)  }}{\sup}  \left \| \chi  \Lambda_{A,q} f   \right \|_{H^{-1/2}(\partial\Omega)}.
\end{equation}

As is well known that in order to obtain stability results one needs \textit{a priori} bounds on the magnetic and electrical potentials (conditional stability), to control oscillations. Thus, for $M>0$ and $\sigma\in\left( 0,1/2\right)$, we define the class of admissible magnetic potentials as
\[
\mathscr{A}(\Omega, M)= \left \{ A \in  W^{2, \infty}(\Omega; \mathbb{R}^n)  :  \left \| A \right \|_{W^{2, \infty}} \leq M \right \},
\]
and the class of admissible electrical potentials as
\[
\mathscr{Q}(\Omega, M, \sigma)= \left \{ q \in H^{\sigma} \cap L^\infty(\Omega; \mathbb{R})  :   \left \| q \right \|_{L^\infty} + \left \| q \right \|_{H^\sigma}\leq M \right \}.
\]
We can now formulate our stability results.
\begin{thm}[Stability for the magnetic potential]\label{SMP}
Let $\Omega \subset \mathbb{R}^n$be a simply-connected open bounded set with smooth boundary. Consider a positive constant $M$. Let $N$ be an open subset of $S^{n-1}$ and consider $F$ an open neighborhood of $F_N$, where $F_N$ is defined as (\ref{FNBN}). Then there exist $C>0$(depending on $n, \Omega, M$) and $\lambda \in \left (0,1/2  \right )$ (depending on $n$) such that the following estimate 
% \left \|d \left [  \chi_\Omega (A_1-A_2)  \right ]  \right \|_{L^2(\mathbb{R}^n)} 
\[
\left \|d  (A_1-A_2)   \right \|_{L^2(\Omega)} \leq C \left | \log \left | \log  \left \|\Lambda^{\sharp}_{1}- \Lambda^{\sharp}_{2}    \right \| \right |   \right |^{-\lambda/2}, 
\]
holds true for all $A_1, A_2 \in \mathscr{A}(\Omega, M)$ satisfying $A_1=A_2$ on $\partial\Omega$; and all $q_1,q_2\in L^{\infty}(\Omega)$. 
\end{thm}

 \begin{thm}[Stability for the electrical potential]\label{SEP}
Let $\Omega \subset \mathbb{R}^n$ be a simply-connected open bounded set with connected smooth boundary and consider two positive constants $M$ and $\sigma\in\left( 0,1/2 \right)$.  Let $N$ be an open subset of $S^{n-1}$ and consider $F$ an open neighborhood of $F_N$, where $F_N$ is defined as (\ref{FNBN}). Then there exist $C>0$(depending on $n, \Omega, M$) and $\lambda \in \left (0,1/2  \right )$ (depending on $n$) such that the following estimate
% \left \|\chi_\Omega (q_1-q_2)  \right \|_{L^2(\mathbb{R}^n)} 
\[
\left \| q_1-q_2  \right \|_{L^2(\Omega)} \leq C \left | \log \left |  \log \left | \log  \left \|\Lambda^{\sharp}_{1}- \Lambda^{\sharp}_{2}    \right \| \right |  \right | \right |^{-\lambda/2}, 
\]
holds true for all $A_1, A_2 \in \mathscr{A}(\Omega, M)$ satisfying $A_1=A_2$ on $\partial\Omega$; and all $q_1, q_2 \in \mathscr{Q}(\Omega, M,\sigma)$.
\end{thm}
\begin{rem}
The simple connectedness hypotheses are needed to use the Carleman estimate derived in \cite{Ch} and the Hodge decomposition obtained in \cite{Tz}. In previous works on stability from infinity, different norms of the Dirichlet-Neumann map have been used. For instance, in \cite{Tz} Tzou considered the DN map $\Lambda: H^{3/2}(\partial\Omega) \rightarrow H^{1/2}(F)$ and the corresponding distance $\left \| \Lambda_1- \Lambda_2  \right \|_{H^{3/2}(\partial\Omega)\rightarrow H^{1/2}(F)}$. Caro, Dos Santos Ferreira and Ruiz \cite{CDSFR} considered the difference of  DN maps $\Lambda_1-\Lambda_2: \mathcal{H}(\partial\Omega)\rightarrow \mathcal{H}(\partial\Omega)^*$, where $ \mathcal{H}(\partial\Omega)$ denotes the range of the trace map $Tr: \left \{u\in L^2(\Omega): \Delta u\in L^2(\Omega) \right \}\rightarrow H^{-1/2}(\partial \Omega)$, and the corresponding distance as:
\[
\left \| \Lambda_1-\Lambda_2 \right \|_{B\rightarrow F} = \underset{Y}{\sup}  \left |  (\left \langle \Lambda_1-  \Lambda_2)u_B, u_F  \right \rangle    \right |,
\]
where
\[
Y=\left \{(u_B, u_F) \in \mathcal{H}(\partial\Omega)^2: \left \| u_B \right \|= \left \| u_F \right \|=1, \, u_B \in \mathcal{E}^\prime(B), \, u_F \in \mathcal{E}^\prime(F)  \right \}.
\]
In the present work we use the natural one (\ref{partialnorm}).
\end{rem}

The proofs of these theorems will be carried out by combining CGO solutions having the support constraint (which will be constructed by using a Carleman estimate with linear weights) and the CGO solutions constructed by Dos Santos Ferreira, Kenig, Sj\"ostrand and Uhlmann \cite{DSFKSjU} (which do not need to have the support constraint). The extra logarithm for the magnetic potentials of the Theorem \ref{SMP} comes from the estimate of the Radon transform obtained in \cite{CDSFR}. To obtain the stability of the electrical potentials, that is Theorem \ref{SEP}, we will use an extra argument. We use the Hodge decomposition derived in \cite{Tz} and the gauge invariance of the DN map in order to use the already established stability estimate for the magnetic potentials. This step involves $\log\log$ of the difference of the partial DN maps, and again by using local estimates for the Radon transform an extra logarithm has to be introduced.\\

%The Carleman estimate with linear weights can be deduced from the Carleman estimate with  logarithmic weights by using an argument due to T. Wolff \cite{Wo}. This argument shows that the Carleman estimate with linear weights (i.e., when we see the illuminated region of $\partial\Omega$ from a given direction $\xi\in S^{n-1}$) can be seen as a limit process of the Carleman estimate with logarithmic weights (i.e., when we see the illuminated region of $\partial\Omega$ from a point $x_0\in \mathbb{R}^n$ which goes to the infinity in the direction $\xi$). Actually this is the motivation to call the Bukgheim and Uhlmann result \cite{BU} as \textit{illuminating} $\Omega$ \textit{from infinity}. 

There is another kind of partial data, the sometimes called local IBVP. The measurements are taken on subsets  of the boundary for functions supported on the same subsets, called the accessible part of the  boundary ($F=B$ in our notation). In this case and in the absence of the magnetic potential ($A\equiv 0$), the    identifiability was obtained by Isakov \cite{Is} assuming that the inaccessible part of the boundary is either part of a plane or a sphere. In this case stability estimates  were obtained by Heck and Wang \cite{HW1} and only requiere a $log$ in the  estimates.  Similarly, Caro \cite{Ca11} derived a $log$-stability estimate for an IBVP with local data for the Maxwell equation under the same flatness condition. We believe that also  our $ log\, log$ stability results should be improved to just one $log$. As it was proved by Mandache in \cite{Ma} the $\log$ is the best stability modulus that one can  expect. Unfortunately for all kinds of partial data, except the mentioned  local problem with restricted geometry, the known stability is $\log \log$  .\\
%and for local data by Heck and Wang \cite{HW1}.\\

%The main tool for all the above mentioned results (partial, local and full data) was the construction of solutions for the equation $\mathcal{L}_{A,q}u=0$, the so-called complex geometric optic solutions (CGO's). These were constructed by using several kind of Carleman estimates, see for example \cite{BU, KSU, DSFKSjU}. \\

This paper is organized as follows. In the section $2$ we prove Theorem \ref{SMP}. In the section $3$ we prove Theorem \ref{SEP}.
\section{Stability estimate for the magnetic potential}

\subsection{CGO solutions and Carleman estimates}

In this section we shall establish the existence of  CGO solutions with the required support constraint for the magnetic Schr\"odinger operator, see Theorem \ref{Ecompact}. For this purpose we introduce some notation. Denote by $Z_N$ the set
\[
Z_N= \underset{\xi \in N }{\bigcup}   \left \{ x\in \partial\Omega  \; : \;  \left \langle \xi, \nu(x)  \right \rangle=0     \right \}
\]
and 
\begin{equation}\label{Ecompact}
E \;  \mbox{be a compact subset of} \;  F_N\setminus Z_N. 
%\mbox{such that} \; \partial \Omega \setminus \overline{B_N}\subset E.  
\end{equation}

\begin{thm}\label{Zs}
Let $\Omega\subset \mathbb{R}^n$ be a bounded open set with smooth boundary. Let $\xi, \zeta  \in S^{n-1}$ be a pair of orthonormal vectors and consider $E$ defined by (\ref{Ecompact}). If $A_1\in C^2(\overline{\Omega}; \mathbb{R}^n)$ and $q_1\in L^\infty(\Omega)$, then there exist two positive constants $\tau_0$ and $C$ (both depending on $n, \Omega, \left \| A_1 \right \|_{C^2}, \left \| q_1 \right \|_{L^\infty}$) such that the equation
 \begin{align*}
     \begin{cases}
            \mathcal{L}_{A_1,q_1}u=0  & \text{ in } \Omega \\ u|_{ E} = 0
     \end{cases}
\end{align*}
has a solution $u_1\in H^1(\Omega)$ of the form
\[
u_1= e^{\tau(\xi\cdot x + i\zeta\cdot x)}\left( e^{\Phi_1} + r_1  \right)-e^{\tau l}b,
\]
with the following properties: 
\begin{item}
\item[(i)] The function  $\Phi_1$ satisfies in $\Omega$
\begin{equation}\label{Phiuno}
(\xi+i\zeta)\cdot \nabla \Phi_1+i (\xi+i\zeta)\cdot A_1 =0.
\end{equation}
and
\begin{equation}
\left \| \Phi_1  \right \|_{W^{\alpha, \infty}}\leq C \left \| A_1 \right \|_{C^{\alpha}(\Omega)} \; , \; \left | \alpha \right |\leq 1.
\end{equation}
%\item[(ii).]  The function $g$ is smooth and satisfies in $\Omega$
%\begin{equation}\label{guno1}
%(\xi+i\zeta)\cdot \nabla g=0.
%\end{equation}
\item[(ii)] The function $l$ depends on the a priori bounds of $A_1$ and $q_1$, and  satisfies
\[
\Re l(x)= \xi\cdot x -k(x), \Im l(x)= \zeta\cdot x + \widetilde{k}(x)
\]
where $k(x) \simeq dist (x, E)$ and $\left | \widetilde{k}(x)  \right | \simeq dist(x, E)$;  in $G$, a neighborhood of $E$ on $\mathbb{R}^n$.
\item[(iii)]  The function $b$ is twice continuously differentiable on $\Omega$ with $\supp b\subset G$; and it  depends on the a priori bounds of $A_1$ and $q_1$.
\item[(iv)] Finally, $r_1\in H^1(\Omega)$ satisfies $r|_{ E}=0$ and for all $\tau \geq \tau_0$ the following estimates hold true
\begin{align*}
 \left \|\partial^{\alpha} r_1  \right \|_{L^2(\Omega)} \leq C \tau ^{ \left | \alpha \right |-1} &, \left |\alpha  \right |	\leq 1,\\
 \left \|  r\right \|_{L^2(\partial\Omega)} \leq C \tau^{-1/2}.
\end{align*}
\end{item}
Moreover, we have 
\begin{equation}\label{bou}
\left \| l  \right \|_{H^1(\Omega)} \leq C, \quad \left \| b  \right \|_{H^1(\Omega)} \leq C
\end{equation}
and
\begin{equation}\label{kapa}
\left \|e^{-\tau k}  \right \|_{L^2(\Omega)} \leq C\tau^{-1/2}, \quad  \left \|e^{-\tau k}  \right \|_{L^\infty(\Omega)} \leq C.
\end{equation}
\end{thm}

\begin{rem}
It is well know that the main tool to derive the existence of CGO solutions for the equation $\mathcal{L}_{A,q}u=0$ is a suitable Carleman estimate with a limiting Carleman weight. We say that a smooth real-valued function $\varphi$ is a limiting Carleman weight, LCW for short, in an open bounded set $\Omega\subset \mathbb{R}^n$ if it has nonvanishing gradient and satisfies pointwise in $\Omega$
\begin{equation}\label{lcww}
\left \langle \varphi^{\prime\prime} \nabla \varphi, \nabla \varphi\right \rangle + \left \langle  \varphi^{\prime\prime} \xi, \xi \right \rangle=0,
\end{equation}
whenever $ \left |  \xi\right |= \left | \nabla\varphi \right |$ and $\nabla\varphi\cdot\xi=0$. It was shown  in \cite{DSFKSU} that there exist only six LCWs for open bounded sets in $\mathbb{R}^n$. In particular we distinguish two LCWs. The linear LCW $\varphi(x)=\xi\cdot x$ where $\xi\in \mathbb{R}^n\setminus \left \{ 0 \right \} $ and the logarithmic LCW $\varphi(x)=\log  \left | x-x_0 \right |$ where $x_0\in \mathbb{R}^n$. Different kinds of Carleman estimates with (and without) boundary terms were obtained by several authors, see for example \cite{BU, DSFKSjU, KSU}.
\end{rem}

To prove Theorem \ref{Zs} we have to obtain a Carleman estimate with a linear LCW  for the magnetic Schr\"odinger operator with the additional support constraint on $E$. More precisely we have the next theorem.

%Before stating our estimate we introduce the $H^1$-semiclassical space with parameter $\tau^{-1}$ denoted by  $H^1_{scl}(\widetilde{\Omega})$. This space endowed with the norm
%\[
%\left \| u \right \|_{H^1_{scl}(\widetilde{\Omega})} = \left \| u \right \|_{L^2(\widetilde{\Omega})} + \tau^{-1} \left \| \nabla u \right \|_{L^2(\widetilde{\Omega})}
%\]
%is a Hilbert space. We denote by $H^{-1}_{scl}(\widetilde{\Omega})$ its dual space. Now we state our Carleman estimate with a linear LCW.

\begin{thm}\label{icar}
Let $\Omega\subset \mathbb{R}^n$ be a simply connected bounded domain with smooth boundary. Consider $\xi\in S^{n-1}$ and $E$ defined by (\ref{Ecompact}).  If $A_1\in C^2(\overline{\Omega}; \mathbb{R}^n)$ and $q_1\in L^\infty(\Omega)$ then given a smooth domain $\widetilde{\Omega}$(depending on $\xi$) satisfying $\Omega\subset \widetilde{\Omega}$, $E\subset \partial \widetilde{\Omega} \cap \partial\Omega$, $ \partial\widetilde{\Omega} \cap \partial\Omega$ be a compact subset of $\partial\Omega_{-,0}(\xi)$; and two positive constants $\widetilde{C}$ and $\tau_0$ (both depending on  $n, \Omega, \left \| A_1 \right \|_{C^2}, \left \| q_1 \right \|_{L^\infty}$)  such that for all $\tau\geq\tau_0$ the following estimate
\begin{equation}\label{lce}
\left \| e^{\tau\xi \cdot x} w   \right \|_{L^2(\Omega)} \leq \widetilde{C}  \left \| e^{\tau \xi\cdot x} \mathcal{L}_{A_1,q_1} w \right \|_{H^{-1}(\widetilde{\Omega})},
\end{equation}
holds true for all $w\in C^{\infty}_0(\Omega)$.
\end{thm}

\begin{rem}
\label{Chuu} For Schr\"odinger operator in the absence of a magnetic potential ($A\equiv 0$) a Carleman estimate adapted to the vanishing condition on the boundary  was obtained by Kening, Sj\"ostrand and Uhlmann, see Proposition $3.2$ in \cite{KSU}. Their estimate are from $L^{2}(\Omega)$ to $L^2(\Omega)$. This is enough to obtain identifiability result because from identity (\ref{al}) it is enough to have boundedness in $L^2(\Omega)$ of the solutions of the Schr\"odinger operator $\mathcal{L}_{0,q}u=0$. In the presence of a magnetic potential  a Carleman estimate adapted to the additional vanishing condition on compact subsets of the boundary was obtained by Chung \cite{Ch} for a logarithmic LCW as follows.
\begin{thm}
Let $\Omega$ be a simply connected bounded domain with smooth boundary. Consider $x_0\in\mathbb{R}^n$ which is not in the closure of the convex hull of $\Omega$.  If $A\in C^2(\overline{\Omega}; \mathbb{R}^n)$ and $q\in L^\infty(\Omega)$ then given a smooth domain $\Omega^\prime$ (depending on $x_0$) satisfying $\Omega\subset \Omega^\prime$ and $\partial\Omega\cap\partial\Omega^\prime$ be a compact subset of $\partial\Omega_{-,0}(x_0)$; there exist two positive constants $C$ and $\tau_0$ (both depending on  $n, \Omega, \left \| A \right \|_{C^2}, \left \| q \right \|_{L^\infty}$)  such that for all $\tau\geq\tau_0$ the following estimate
\[
\left \|  \left | x-x_0 \right |^\tau w \right \|_{L^2(\Omega)} \leq C \left \| \left | x-x_0 \right |^\tau \mathcal{L}_{A,q}w \right \|_{H^{-1}(\Omega^\prime)},
\]
holds true for all $w\in C^{\infty}_0(\Omega)$.
\end{thm}
Roughly speaking we explain Chung's arguments to prove the above estimate. Let $\Omega$ be an open simply-connected set with smooth boundary. Consider $x_0\in \mathbb{R}^n$ which is not in the closure of the convex hull of $\Omega$.  Without loss of generality we can assume that $x_0=0$. Then we can separate the point $x_0$ and $\Omega$ by a plane. Actually he assumed that $\Omega$ lies entirely in the half-space $E_n:=\left \{(x_1, \ldots, x_n): x_n>0  \right \}$ and that there exists $r_0>0$ such that $B_{r_0}(0)\cap \Omega = \emptyset$. In this setting he introduced a change of variables in $(\mathbb{R}^n\setminus B_{r_0}(0)) \cap E_n$ given by 
\[
\Theta: (\mathbb{R}^n\setminus B_{r_0}(0)) \cap E_n \rightarrow \mathbb{R}^+\times S^{n-1}, \quad \Theta(x)= (r, \theta),
\]
where $r=\log \left |x-x_0\right | $ and $\theta= (x-x_0)/ \left | x-x_0 \right |$. As a first approximation, from the illuminated condition of $\partial\Omega_{-,0}(x_0)$, the set $E$ can be defined by a graph $r=f(\theta)$, where $f:S^{n-1}\rightarrow \left(r_0, \infty  \right)$ is smooth enough. Then by making another change of variables $\widetilde{\Theta}: \mathbb{R}^+\times S^{n-1} \rightarrow S^1\times S^{n-1}$ defined by $\widetilde{\Theta}(r, \theta)=(r/f(\theta),\theta)$, we can see $E$ as a part of the unit sphere $S^{n-1}$ and then by a partition of unity and local changes of variables he obtained flatness on $E$ on the $r$-variable. This is the main point to derive his Carleman estimate.
%Later his computations involve a Laplace-Beltrami operator (with lower perturbation) in the $(r,\theta)$-variables and the construction of a semiclassical pseudo-differential operator $J$ of order $1$, called by him as \textit{magic operator}. This operator has nice properties. In particular preserves the support of the functions. This facts have as consequence the well behavior (it means preserve the support) of the commutator between $\mathcal{L}_{A,q}$ and $J$. This is the main point to derive his Carleman estimate.
In our case to prove Theorem \ref{icar} we can follow similar arguments described above. Without loss of generality we can assume that $\xi=e_n$, where $e_n$ is the $n$-th canonical unit vector in $\mathbb{R}^n$. We assume that $\Omega$ lies entirely in the half-space $E_{r_0}:=\left \{(x_1, \ldots, x_n): x_n>r_0  \right \}$ for some positive constant $r_0$. Now from the illuminated condition of $\partial\Omega_{-,0}(\xi)$, the set $E$ can be defined by a graph $x_n=f(x^\prime)$, where $x=(x^\prime, x_n)\in \mathbb{R}^n$ and $x^\prime=(x_1, \ldots, x_{n-1})$. Notice in particular that $f(x^\prime)>r_0$. Next we introduce the change of variables $\Theta: E_{r_0} \rightarrow  \left \{(x_1, \ldots, x_n)\in \mathbb{R}^n: x_n=1  \right \}$ defined by $\Theta(x^\prime, x_n)= (x^\prime, x_n/f(x^\prime))$. Thus we can see $E$ as a subset of the hyperplane  $ \left \{(x_1, \ldots, x_n)\in \mathbb{R}^n: x_n=1  \right \}$ and then by a partition of unity and local changes of variables we shall obtain flatness on $E$ on the $x_n$-variable. From here we can use analogous arguments to the ones used by Chung with the difference that our computations are easier than in his case because we have a linear LCW $\varphi(x)= x\cdot e_n=x_n$. The repetition of his arguments would be large and tedious and we hope that the reader would be convinced by our heuristic justification. 
\end{rem}
We mention that in the global case, that say $\Omega=\widetilde{\Omega}=\mathbb{R}^n$,  there exists an argument which allows to obtain a Carleman estimate with linear weight from another with a logarithmic weight. This argument is due to Wolff \cite{Wo} and was used by him in the context of unique continuation for second order elliptic operators to deduce the Kenig, Ruiz and Sogge estimate \cite{KRS} from the Jerison-Kenig estimate \cite{JK}. We can use as well this kind of argument to obtain the Carleman estimate (\ref{icar}) from the Carleman estimate with logarithmic weight obtained by Chung \cite{Ch}. The point is to track back all the constants which appearing in Chung's estimates and see the $\xi$-illuminated face  $\partial\Omega_{-,0}(\xi)$ as a limit of illuminated sets from a point  $\partial\Omega_{-,0}(x_0)$ when $x_0$ goes to the infinity in the direction $\xi$. 
\begin{rem}[About Theorem \ref{Zs}] The proof of Theorem \ref{Zs} follows by combining standard arguments. As first step we use estimate (\ref{lce}) and Hahn-Banach theorem to obtain the following result. For every $v\in L^{2}(\Omega)$, there exists $r\in H^{1}_0(\widetilde{\Omega})$ satisfying in $\Omega$
\begin{equation}\label{duals}
\left( e^{-\tau x\cdot \xi} \tau^{-2} \mathcal{L}_{A_1,q_1} e^{\tau x\cdot \xi}  \right)r=v.
\end{equation}
Hence its trace vanishes on $E\subset \partial\widetilde{\Omega}$. Moreover, such solution satisfies
\[
\left \| r \right \|_{H^1(\Omega)}\leq C\left \| v \right \|_{L^2(\Omega)},
\]
where $C$ is a positive constants depending on $n, \Omega, \left \| A_1 \right \|_{C^2}$ and $ \left \| q_1 \right \|_{L^\infty}$. Remember that we are looking for solutions of the magnetic Schr\"odinger operator $\mathcal{L}_{A_1,q_1}u_1=0$ of the form
\[
u_1= e^{\tau(\xi\cdot x + i\zeta\cdot x)}\left( e^{\Phi_1} + r_1  \right)-e^{\tau l}b.
\]
The function $\Phi_1$ satisfies equation (\ref{Phiuno}) which can be solved by taking into account the operator $(\xi+i\zeta)\cdot \nabla$ as a $\overline{\partial}$-operator in a suitable variables. The existence of the function $r_1$ with the corresponding estimates is given by (\ref{duals}) for some suitable $v\in L^{2}(\Omega)$. Finally the existence of the functions $l$ and $b$ can be followed by similar arguments from Proposition $9.2$ in \cite{Ch}.
\end{rem}

We will also use the solutions constructed by Dos Santos Ferreira, Kenig, Sj\"ostrand and Uhlmann, see Lemma $3.4$ in \cite{DSFKSjU}. These solutions do not require the support constraint on $E$.

\begin{thm}\label{DSFKSjUs}
Let $\Omega\subset \mathbb{R}^n$ be a bounded open set with smooth boundary. Let $\xi, \zeta  \in S^{n-1}$ be a pair of orthonormal vectors. If $A_2\in C^2(\overline{\Omega}; \mathbb{R}^n)$ and $q_2\in L^\infty(\Omega)$ then there exist two positive constants $\tau_0$ and $C$ (both depending on $n, \Omega,\left \| A_2 \right \|_{C^2}, \left \| q_2 \right \|_{L^\infty} $) such that the equation  $\mathcal{L}_{A_2,\overline{q}_2}u=0$ has a solution $u_2\in H^1(\Omega)$ of the form
\[
u_2= e^{-\tau(\xi\cdot x -i \zeta\cdot x)}\left( e^{\Phi_2}g + r_2  \right),
\]
with the following properties: 
\begin{item}
\item[(i)] The function  $\Phi_2$ satisfies in $\Omega$
\begin{equation}\label{Phidos}
(\xi+i\zeta)\cdot \nabla \overline{\Phi}_2-i (\xi+i\zeta)\cdot A_2 =0.
\end{equation}
and
\begin{equation}
\left \| \Phi_2  \right \|_{W^{\alpha, \infty}}\leq C \left \| A_2 \right \|_{C^{\alpha}(\Omega)} \; , \; \left | \alpha \right |\leq 1.
\end{equation}
\item[(ii)]  The function $g$ is smooth and satisfies in $\Omega$
\begin{equation}\label{guno1}
(\xi+i\zeta)\cdot \nabla g=0.
\end{equation}
\item[(iii)] The  function $r_2$ belongs to $H^1(\Omega)$ and satisfies the following estimate
\[
 \left \|\partial^{\alpha} r_2  \right \|_{L^2(\Omega)} \leq C \tau ^{ \left | \alpha \right |-1} \left \|  g\right \|_{H^2(\Omega)},\quad \left |\alpha  \right |	\leq 1,
\]
for all $\tau\geq \tau_0$.
\end{item}
\end{thm}

\begin{rem}\label{minuszeta}
We mention that Theorem \ref{DSFKSjUs} was stated for any LCW. From condition (\ref{lcww}), if $\varphi$ is a LCW then $-\varphi$ is also a LCW. As a consequence Theorems \ref{Zs} and \ref{DSFKSjUs} remain true replacing $\xi\cdot x$ by $-\xi\cdot x$; and we have analogous estimates for the respective solutions $u_1$ and $u_2$. 
\end{rem}

\begin{rem}\label{infi}
The following result was proved in \cite{SaM} (see Lemma $4.6$ in \cite{SaM} and also Lema $2.1$ in \cite{Sun}). Let $\xi_0\in \mathbb{C}^n$ such that $\Re \xi_0\cdot \Im \xi_0=0$ and $\left | \Re \xi_0 \right | = \left |\Im \xi_0  \right |=1$. If $W\in L^{\infty}(\mathbb{R}^n)$ then there exists a solution  $\Phi\in  L^{\infty}(\mathbb{R}^n)$ of the equation  
\[
\xi_0\cdot \nabla \Phi + i\xi_0 \cdot W =0.
\]
Moreover, there exists $C>0$ such that  
\begin{equation}\label{infinito}
\left \| \Phi \right \|_{L^\infty(\mathbb{R}^n)} \leq C \left \| W \right \|_{L^\infty(\mathbb{R}^n)}.
\end{equation}
\end{rem}

The following proposition concerns Carleman estimates for the magnetic Schr\"odinger operator and was proved by Dos Santos Ferreira, Kenig, Sj\"ostrand and Uhlmann, see Proposition $2$ in  \cite{DSFKSjU}.

%We will use it (in the next section) in order to estimate the DN map measurement on one side of the boundary, namely $\Omega_{+,0}(\xi)$, from the other side $\Omega_{-,0}(\xi)$. 

\begin{prop}[A Carleman estimate with boundary terms]  \label{PCe}
Let $\Omega\subset\mathbb{R}^n$ be a bounded open set with smooth boundary. Let $\xi \in S^{n-1}$ and define $\varphi(x)=\xi\cdot x$.  If $A\in C^1(\overline{\Omega}; \mathbb{R}^n)$ and $q\in L^\infty(\Omega; \mathbb{R})$ then there exist two positive constants $\tau_0>0$ and $C>0$ (both depending on $n, \Omega,\left \| A \right \|_{C^1}, \left \| q \right \|_{L^\infty} $) such that  for all $u\in C^{\infty}(\overline{\Omega})\cap H^1_0(\Omega) $ the following estimate holds true  for all $\tau \geq \tau_0$
\begin{equation}\label{Ce}
\begin{aligned}
&\left \| \sqrt{\partial_\nu \varphi } \: e^{\tau\varphi} \partial_\nu u \right \|_{L^2(\Omega_{+,0}(\xi))} +\tau^{1/2} \left \| e^{\tau\varphi} u  \right \|_{L^2(\Omega)} + \tau^{-1/2} \left \| e^{\tau\varphi} \nabla u  \right \|_{L^2(\Omega)}\\
&\qquad \leq C\left(   \tau^{-1/2} \left \| e^{\tau\varphi} \mathcal{L}_{A,q} u \right \|_{L^2(\Omega)} + \left \| \sqrt{-\partial_\nu \varphi } \:e^{\tau\varphi} \partial_\nu u \right \|_{L^2(\Omega_{-,0}(\xi))}  \right),
\end{aligned}
\end{equation}
where $\nu$ denotes the exterior unit normal of $\partial\Omega$ and $\partial_\nu= \nu\cdot \nabla$. 
\end{prop}
\begin{rem}\label{Cer}
The Carleman estimate (\ref{Ce}) is still true for all $u$ in $H^1_0(\Omega)$ such that $\mathcal{L}_{A,q}u \in L^2(\Omega)$. This could be seen by a standard regularization method. The vanishing of the trace of the function $u$ is essential for this estimate. Notice that in the above inequality we bound the $L^2(\Omega_{+,0}(\xi))$-norm by the $L^2(\Omega_{-,0}(\xi))$-norm plus remainder terms in $L^2(\Omega)$-norm. In other words we bound the unknown measurements of the shadow face of $\partial\Omega$ by know measurements of the  illuminated face but we have to pay with remainder terms in $L^2(\Omega)$-norm. This fact will be useful in our approach.
\end{rem}

\subsection{From the boundary to the interior}

The following lemma was proved in \cite{DSFKSjU}. 
\begin{lem}\label{gmagneticf}
Let $\Omega \subset \mathbb{R}^n$ be a bounded open set with smooth boundary. If $A\in C^1(\overline{\Omega}; \mathbb{R}^n)$ and $q\in L^{\infty}(\Omega)$ then for all $u,v$ in $L^2(\Omega)$ such that  $\Delta u,\Delta v\in L^2{(\Omega)} $, we have the magnetic Green formula
\begin{equation}
\begin{aligned}
& \left \langle  \mathcal{L}_{A,q}u,v   \right \rangle_{L^2(\Omega} - \left \langle u, \mathcal{L}_{A,\overline{q}}v \right \rangle_{L^2(\Omega)}\\
&  \qquad \qquad = \left \langle u, (\partial_\nu + i\nu\cdot A)v \right \rangle_{L^2(\partial \Omega)} - \left \langle  (\partial_\nu + i\nu\cdot A)u, v  \right \rangle_{L^2(\partial \Omega)}.
\end{aligned}
\end{equation}\end{lem}

\begin{lem}
Let $\Omega \subset \mathbb{R}^n$ be a bounded open set with smooth boundary. If $A_1, A_2\in C^2(\overline{\Omega}; \mathbb{R}^n)$ and $q_1, q_2\in L^{\infty}(\Omega)$ then
\begin{equation}\label{al}
\begin{aligned}
& \left \langle  (\Lambda_1 - \Lambda_2)u_1,u_2 \right \rangle_{L^2(\partial\Omega)}\\
&=\int_{\Omega} \left[ (A_1-A_2)\cdot(Du_1 \overline{u}_2 + u_1 \overline{Du}_2) + (A_1^2-A_2^2+q_1-q_2)u_1\overline{u}_2\right]  ,
 \end{aligned}
\end{equation}
where $u_1, u_2 \in H^{1}(\Omega)$ satisfying $\mathcal{L}_{A_1,q_1}u_1=0$ and $\mathcal{L}_{A_2, \overline{q}_2}u_2=0$ in $\Omega$.
\end{lem}
\begin{proof}
The proof of this lemma was implicit in section $4$ of \cite{DSFKSjU}, and only for completeness we prove it following their ideas. Let $u_1, u_2 \in H^{1}(\Omega)$ such that $\mathcal{L}_{A_1,q_1}u_1=0$ and $\mathcal{L}_{A_2, \overline{q}_2}u_2=0$ in $\Omega$. We introduce an auxiliary function $w$ satisfying in $\Omega$
\begin{align}
\label{artificial}
     \begin{cases}
            \mathcal{L}_{A_2,q_2}w=0,  & 
            \\ w|_{\d \Omega} = u_1|_{\d \Omega} .
     \end{cases}
\end{align}
Thus, by the definition of the DN map we get
\begin{equation}\label{ale1}
\begin{aligned}
&\left \langle  (\Lambda_1 - \Lambda_2)u_1,u_2 \right \rangle_{L^2(\partial\Omega)} \\
& \qquad =\left \langle (\partial\nu+ i\nu\cdot A_1)u_1, u_2 \right \rangle - \left \langle (\partial\nu+ i\nu\cdot A_2)w, u_2 \right \rangle_{L^2(\partial\Omega)}\\
&\qquad= \left \langle \partial_\nu(u_1-w)+ i\nu\cdot (A_1-A_2)u_1,u_2 \right \rangle_{L^2(\partial\Omega)}.
\end{aligned}
\end{equation}
On the other hand, we compute $ \left \langle  \mathcal{L}_{A_2,q_2}(w-u_1), u_2  \right \rangle_{L^2(\Omega)}$ in two different ways. For the first one, we use Lemma \ref{gmagneticf} and (\ref{artificial})-(\ref{ale1}) to obtain
\begin{equation}\label{ale2}
\begin{aligned}
& \left \langle  \mathcal{L}_{A_2,q_2}(w-u_1), u_2  \right \rangle_{L^2(\Omega)}\\
& = \left \langle w-u_1, \mathcal{L}_{A_2, \overline{q}_2}u_2  \right \rangle_{L^2(\Omega)} + \left \langle w-u_1, (\partial_{\nu} + i\nu\cdot A_2)u_2 \right \rangle_{L^2(\partial\Omega)} \\
&\qquad   - \left \langle (\partial_{\nu} + i\nu\cdot A_2)(w-u_1), u_2 \right \rangle_{L^2(\partial\Omega)} \\
&=  \left \langle \partial_\nu(u_1-w),u_2 \right \rangle_{L^2(\partial\Omega)}\\
& = \left \langle  (\Lambda_1 - \Lambda_2)u_1,u_2 \right \rangle_{L^2(\partial\Omega)} - i \left \langle  \nu\cdot (A_1-A_2)u_1,u_2 \right \rangle_{L^2(\partial\Omega)} .
\end{aligned}
\end{equation}
For the second, again from (\ref{artificial}) and integration by parts we have
%\begin{equation}\label{ale3}
\begin{align*}
& \left \langle  \mathcal{L}_{A_2,q_2}(w-u_1), u_2  \right \rangle_{L^2(\Omega)} \\
&= \left \langle ( \mathcal{L}_{A_1,q_1} - \mathcal{L}_{A_2,q_2})u_1, u_2  \right \rangle_{L^2(\Omega)}\\
&= \left \langle  (A_1-A_2)\cdot Du_1 + D\cdot \left( (A_1-A_2)u_1 \right) + (A_1^2-A_2^2+q_1-q_2)u_1,u_2   \right \rangle_{L^2(\Omega)}\\
&= \int_{\Omega} \left( (A_1-A_2)\cdot(Du_1 \overline{u}_2 + u_1 \overline{Du}_2) + (A_1^2-A_2^2+q_1-q_2)u_1\overline{u}_2  \right)dx\\
& \qquad \qquad -i   \left \langle  \nu\cdot (A_1-A_2)u_1,u_2 \right \rangle_{L^2(\partial\Omega)}.
\end{align*}
Combining the above equality with (\ref{ale2}), we conclude the proof.
%\end{equation}
\end{proof}

The identity \ref{al} is called Alessandrini's identity and it was used by him to obtain a stability estimate for Calder\'on's problem \cite{A}. The idea to prove stability estimates for magnetic potentials is as follows:  we plug the solutions constructed for the magnetic Schr\"odinger operator $\mathcal{L}_{A,q}$ (see Theorems \ref{Zs} and \ref{DSFKSjUs})  into (\ref{al}) and then compute both left and right hand sides of the Alessandrini's identity separately. The left hand side gives us an estimate which involves the difference between the partial DN maps $ \Lambda_1^\sharp- \Lambda_2^\sharp $, and in computing the right hand side it will appear naturally the Radon transform of  $\chi_{\Omega}(A_1-A_2)$, where $\chi_\Omega$ denotes the characteristic function of $\Omega$. Finally combining both estimates and using estimates for the Radon transform obtained in \cite{CDSFR}, we will prove Theorem \ref{SMP}. The remainder of this section will be devoted to develop these ideas.
\begin{rem}
For technical reasons we introduce some constants.  Consider 
\begin{equation}\label{Ob}
c= \underset{x\in \Omega }{\sup}  \left |\xi\cdot x  \right |, \quad  \xi\in S^{n-1},
\end{equation}
($c$ is finite since $\Omega$ is bounded)  and $\epsilon>0$ small enough such that
\begin{equation}\label{Fn1}
F_N \subset F_{N, \epsilon}\subset\subset F,
\end{equation}
where
\begin{equation}\label{FN}
F_{N, \epsilon}= \underset{\xi\in N}\bigcup \Omega_{-,\epsilon}(\xi)
\end{equation}
and the set $F_N$ is defined in (\ref{FNBN}). In this setting, let $\chi\in C^\infty(\partial \Omega)$ be a cutoff function  supported in $F$ such that it is equals to $1$ on $F_{N, \epsilon}$.
\end{rem}

\begin{prop}\label{leftsideAI} Let $\Omega\subset\mathbb{R}^n$ be a bounded open set with smooth boundary. Consider two positive constants $c$ given by (\ref{Ob}) and $\epsilon$ satisfy  (\ref{Fn1}). Let $M>0$ and consider  $A_1, A_2\in \mathscr{A}(\Omega, M)$ and $q_1, q_2\in L^\infty(\Omega)$. If $u_1, u_2\in H^1(\Omega)$ such that $\mathcal{L}_{A_1, q_1}u_1=0$ and $\mathcal{L}_{A_2, \overline{q}_2}u_2=0$ then there exist two positive constants $\tau_0>0$ and $C>0$ (both depending on $n, \Omega, M, \epsilon$) such that  the estimate

\begin{equation}\label{aleinequality}
\begin{aligned}
&\left |  \left \langle (\Lambda_1- \Lambda_2)u_1, u_2  \right \rangle_{L^2(\partial\Omega)} \right |\\
& \leq  C\left \| \chi ( \Lambda_1 - \Lambda_2)  \right \|  \left(  \left \| u_1 \right \|_{H^1(\Omega)}  \left \| u_2 \right \|_{H^1(\Omega)} \right.\\
&\qquad \qquad  \left.  + e^{\tau c}\left \| u_1 \right \|_{H^1(\Omega)}   \left \| e^{\tau \xi \cdot x}u_2 \right \|_{L^{2}(\partial\Omega)} \right)  \\
&  \qquad+ C  \tau^{-\frac{1}{2}} \left \| e^{-\tau \xi \cdot x}(\mathcal{L}_{A_1,q_1} - \mathcal{L}_{A_2,q_2}) u_1 \right \|_{L^2(\Omega)}  \left \| e^{\tau \xi \cdot x} u_2 \right \|_{L^{2}(\partial\Omega)}\\
&\qquad  + C \left \| e^{-\tau \xi \cdot x}u_1 \right \|_{L^{2}(\partial\Omega)} \left \| e^{\tau \xi \cdot x} u_2 \right \|_{L^{2}(\partial\Omega)}
\end{aligned}
\end{equation}
holds true  for all $\tau \geq \tau_0$ and all $\xi\in N$.
\end{prop}

\begin{proof}   We begin by denoting $\Lambda_{A_i, q_i}= \Lambda_i$ for $i=1,2$. Let us decompose the difference between the DN maps in the following way
\[
\Lambda_1 - \Lambda_2 = \chi (\Lambda_1 - \Lambda_2)+ (1-\chi) (\Lambda_1 - \Lambda_2).
\]
Thus we have
\begin{equation}\label{decomposition}
\begin{aligned}
&\left \langle (\Lambda_1- \Lambda_2)u_1, u_2  \right \rangle_{L^2(\partial\Omega)} = \left \langle\chi (\Lambda_1- \Lambda_2)u_1, u_2  \right \rangle_{L^2(\partial\Omega)}   \\
&\qquad \qquad \qquad \qquad \qquad \qquad \;  \; +  \left \langle(1-\chi) (\Lambda_1- \Lambda_2)u_1, u_2  \right \rangle_{L^2(\partial \Omega)}
\end{aligned}
\end{equation}

We now estimate each term of the previous summation. For the first term, the Cauchy-Schwarz inequality gives
\begin{equation}\label{unoxx}
\begin{aligned}
\left |  \int_{\partial\Omega}\chi   (\Lambda_1- \Lambda_2)u_1\overline{u}_2 dS \right |&  \leq   \left \| \chi (\Lambda_1 - \Lambda_2) \right \| \left \| u_1 \right \|_{H^{\frac{1}{2}}(\partial \Omega)} \left \| u_2 \right \|_{H^{\frac{1}{2}}(\partial \Omega)}\\
&\leq \left \| \chi (\Lambda_1 - \Lambda_2) \right \| \left \| u_1 \right \|_{H^1(\Omega)}  \left \| u_2 \right \|_{H^1(\Omega)}.
\end{aligned}
\end{equation}

The second term requires a more refined analysis. Let $w$ be a function such that it satisfies (\ref{artificial}). Then for every $\xi \in N$ we get

%%%%%%%%%%%%%%%%%%%%%%%%%%%%%%%%%%%%%

\begin{equation}\label{secondtermt}
\begin{aligned}
&\left |  \int_{\partial\Omega}(1-\chi)   (\Lambda_1- \Lambda_2)u_1\overline{u}_2 dS \right | \\
& = \left |  \int_{\Omega_{-,\epsilon}(\xi)\cup (\partial\Omega\setminus \Omega_{-,\epsilon}(\xi))}(1-\chi)   (\Lambda_1- \Lambda_2)u_1\overline{u}_2 dS \right | \\
& = \left |  \int_{\partial\Omega\setminus \Omega_{-,\epsilon}(\xi)} (1-\chi)   (\Lambda_1- \Lambda_2)u_1\overline{u}_2 dS  \right | \\
&\leq  C_1 \left \| e^{-\tau \xi\cdot x}  (\Lambda_1- \Lambda_2)u_1 \right \|_{L^2(\partial\Omega\setminus\Omega_{-,\epsilon}(\xi))}  \left \| e^{\tau \xi\cdot x} u_2\right \|_{L^2(\partial\Omega\setminus \Omega_{-,\epsilon}(\xi))}.
\end{aligned}
\end{equation}
We next turn to the $L^2(\partial\Omega\setminus \Omega_{-,\epsilon}(\xi))$-norms in the previous inequality. Since $u_1\in H^1(\Omega)$ and $\mathcal{L}_{A_2,q_2}(w-u_1)= (\mathcal{L}_{A_1,q_1} - \mathcal{L}_{A_2,q_2}) u_1$ (where $w$ is the auxiliary function in (\ref{artificial})), it follows that $\mathcal{L}_{A_2,q_2}(w-u_1)\in L^2(\Omega)$. Moreover, we have that $w-u_1\in H^1_0(\Omega)$. Hence the Carleman estimate (\ref{Ce}) from Proposition \ref{PCe} and Remark \ref{Cer} imply that
\begin{equation}\label{firststepw}
\begin{aligned}
& \left \| e^{-\tau \xi\cdot x}   (\Lambda_1- \Lambda_2)u_1 \right \|_{L^2(\partial\Omega\setminus \Omega_{-,\epsilon}(\xi))} \\
& =  \left\|    e^{-\tau \xi\cdot x}   \left( \partial_\nu(u_1-w) + i\nu\cdot (A_1-A_2)u_1   \right) \right\|_{L^2(\partial\Omega\setminus \Omega_{-,\epsilon}(\xi))}  \\
& \leq \left\|    e^{-\tau \xi\cdot x}    \partial_\nu(u_1-w) \right\|_{L^2(\partial\Omega\setminus \Omega_{-,\epsilon}(\xi))}   + C_1 \left\|    e^{-\tau \xi\cdot x}   u_1 \right\|_{L^2(\partial\Omega\setminus \Omega_{-,\epsilon}(\xi))} \\
& \leq  \dfrac{1}{\sqrt{\epsilon}}    \left\|   \sqrt{\left\langle \xi \cdot \nu(\cdot)  \right \rangle}    e^{-\tau \xi\cdot x}    \partial_\nu(u_1-w) \right\|_{L^2(\partial\Omega\setminus \Omega_{-,\epsilon}(\xi))}  \\
&\qquad \qquad + C_1 \left\|    e^{-\tau \xi\cdot x}   u_1 \right\|_{L^2(\partial\Omega\setminus \Omega_{-,\epsilon}(\xi))} \\
&  \leq  \dfrac{1}{\sqrt{\epsilon}}    \left\|   \sqrt{\left\langle \xi \cdot \nu(\cdot)  \right \rangle}    e^{-\tau \xi\cdot x}    \partial_\nu(u_1-w) \right\|_{L^2( \Omega_{+,0}(\xi))} \\
&\qquad \qquad+ C_1 \left\|    e^{-\tau \xi\cdot x}  u_1 \right\|_{L^2(\partial\Omega)}  \\
&  \leq  \dfrac{C_2}{\sqrt{\epsilon}}\left( \left \|  \sqrt{-\left\langle \xi \cdot \nu(\cdot)  \right \rangle} e^{-\tau \xi\cdot x}  \partial_{\nu}(u_1-w) \right \|_{L^2(\partial\Omega_{-,0}(\xi))}  \right.\\
& \;\; \; \;   \left. + \tau^{-\frac{1}{2}} \left \| e^{-\tau \xi\cdot x}   \mathcal{L}_{A_2,q_2} (w-u_1) \right \|_{L^2(\Omega)}    \right) +C_1 \left.  \left\|    e^{-\tau \xi\cdot x}   u_1 \right\|_{L^2(\partial\Omega)} \right.\\
&   \leq  \dfrac{C_2}{\sqrt{\epsilon}}\left( \left \| e^{-\tau \xi\cdot x}  \partial_{\nu}(u_1-w) \right \|_{L^2(\partial\Omega_{-,0}(\xi))}  \right.\\
& \;\; \; \;   \left. + \tau^{-\frac{1}{2}} \left \| e^{-\tau \xi\cdot x}   (\mathcal{L}_{A_1,q_1} - \mathcal{L}_{A_2,q_2}) u_1 \right \|_{L^2(\Omega)}    \right) +C_1 \left.  \left\|    e^{-\tau \xi\cdot x}   u_1 \right\|_{L^2(\partial\Omega)}.\right.
%& \qquad \qquad \qquad  \qquad \qquad  \qquad  \qquad \qquad \qquad + \left.  \left\|    e^{-\tau \xi\cdot x}   u_1 \right\|_{L^2(\partial\Omega)}  \right)
\end{aligned}
\end{equation}
Now we estimate the $L^2(\partial\Omega_{-,0}(\xi))$-norm in the last inequality as follows:
 \begin{equation}\label{secondstepw}
\begin{aligned}
&  \left \|  e^{-\tau \xi\cdot x}  \partial_{\nu}(u_1-w) \right \|_{L^2(\partial\Omega_{-,0}(\xi))}  \\
& =   \left \| e^{-\tau \xi\cdot x}\left [ (\Lambda_1-\Lambda_2)u_1 -i \nu\cdot(A_1-A_2)u_1 \right ] \right \|_{L^2(\partial\Omega_{-,0}(\xi))}    \\
%& \leq  \left \| e^{-\tau \xi\cdot x} (\Lambda_1-\Lambda_2)u_1  \right \|_{L^2(\partial\Omega_{-,0}(\xi))}   \\
%& \qquad +  \left \| e^{-\tau \xi\cdot x} i \nu\cdot(A_1-A_2)u_1  \right \|_{L^2(\partial\Omega_{-,0}(\xi))}  \\
& \leq  \left \| e^{-\tau \xi\cdot x} \chi (\Lambda_1-\Lambda_2)u_1  \right \|_{L^2(\partial\Omega)}  \\
& \qquad +  \left \| e^{-\tau \xi\cdot x} i \nu\cdot(A_1-A_2)u_1  \right \|_{L^2(\partial\Omega_{-,0}(\xi))}  \\
& \leq  e^{\tau k}  \left \| \chi (\Lambda_1 - \Lambda_2) \right \| \left \| u_1 \right \|_{H^{\frac{1}{2}}(\partial \Omega)}  +  C_3\left \| e^{-\tau \xi\cdot x}  u_1  \right \|_{L^2(\partial\Omega)}. 
\end{aligned}
\end{equation}
Thus, replacing (\ref{firststepw}) and  (\ref{secondstepw}) into  (\ref{secondtermt}) gives us
\begin{equation} \label{unox}
\begin{aligned}
& \left |  \int_{\partial\Omega}(1-\chi)   (\Lambda_1- \Lambda_2)u_1\overline{u}_2 dS \right | \\
& \leq C_4\left(   \epsilon^{-1/2}e^{\tau k}   \left \| \chi (\Lambda_1 - \Lambda_2) \right \| \left \| u_1 \right \|_{H^1(\Omega)}                                     \right.\\
& \qquad \qquad \qquad  +\epsilon^{-1/2} \tau^{-\frac{1}{2}} \left \| e^{-\tau \xi\cdot x}(\mathcal{L}_{A_1,q_1} - \mathcal{L}_{A_2,q_2}) u_1 \right \|_{L^2(\Omega)}  \\
& \qquad \qquad \qquad  \qquad \qquad \qquad   \left.   + \left \| e^{-\tau \xi\cdot x} u_1  \right \|_{H^1(\Omega)}  \right) \left \| e^{\tau \xi\cdot x} u_2\right \|_{L^2(\partial\Omega)}. 
\end{aligned}
\end{equation}
Finally we conclude the proof replacing (\ref{unoxx}) and (\ref{unox}) into (\ref{decomposition}).
\end{proof}

%%%%%%%%%%%%%%%%%%%%%%%%%%%%%%%%%%%%%%%%%%%%%%%%%%%%%%%%%%%%%%%%%%%%%%%%%%%%%%%%%%%%%%%%%%%%%%%%%%%%%%%%%%%%

\begin{cor} \label{righthandside}Let $\Omega\subset \mathbb{R}^n$ be a bounded open set with smooth boundary. Let $M>0$ and consider $A_1, A_2\in \mathscr{A}(\Omega, M)$ and $q_1, q_2\in L^\infty(\Omega)$. Let $u_1\in H^1(\Omega)$ be a solution of  $\mathcal{L}_{A_1, q_1}u=0$ constructed in Theorem  \ref{Zs} and let $u_2\in H^1(\Omega)$ be a solution of  $\mathcal{L}_{A_2, q_2}u=0$ constructed in Theorem  \ref{DSFKSjUs}. Then there exist $\tau_0>0$ and $C>0$ (both depending on $n, \Omega, M$) such that the estimate
\begin{equation}\label{dnmap}
\begin{aligned}
& \tau^{-1} \left |  \left \langle (\Lambda_1- \Lambda_2)u_1, u_2  \right \rangle_{L^2(\partial\Omega)} \right |\\
&\qquad \leq C \left(  e^{4\tau c}\left \| \Lambda_1^\sharp- \Lambda_2^\sharp \right \|+\tau^{-\frac{1}{2}} \right)  \left \|\overline{g} \right \|_{H^2(\Omega)}
\end{aligned}
\end{equation}
holds true for all $\tau \geq \tau_0$.
\end{cor}
\begin{proof} We start by computing the norms corresponding to $u_1$ in the right hand side of (\ref{aleinequality}). The estimates for $u_2$ are similar. By Theorem \ref{Zs}, the function $u_1$ has the form
\[
u_1= e^{\tau(\xi\cdot x + \zeta\cdot x)}\left( e^{\Phi_1} + r_1  \right)-e^{\tau l}b
\]
and there exist  two positive constants $C_1$ and $\tau_1$ such that the following estimate 
\begin{equation}\label{estar1}
%\begin{aligned}
%& \left \| e^{\Phi_1 }g \right \|_{H^1(\Omega)}\leq C_1 \left \| g \right \|_{H^1(\Omega))} ,\\
 \left \| \partial^{\alpha} r_1 \right \|_{L^2(\Omega)}\leq C_1 \tau^{\left | \alpha \right |-1}, \quad \left | \alpha \right |\leq 1,
%  \left \| \partial^{\alpha} r_2 \right \|_{L^2(\Omega)}\leq C_0 h ^{1- \left | \alpha \right |} \; , \; \left | \alpha \right |\leq 1,
%\end{aligned}
\end{equation}
%\end{equation}
holds true for all $\tau \geq \tau_1$. Also, we have the estimate
\[
 \left \| e^{\tau l } \right \|_{L^\infty(\Omega)} =  \left \| e^{\tau (\xi\cdot x -k(x)) } \right \|_{L^\infty(\Omega)} \leq  \left \| e^{\tau \xi\cdot x } \right \|_{L^\infty(\Omega)}\leq  e^{\tau c}.
\]
For convenience we denote 
\begin{equation}\label{aaaaa5}
a_1=e^{\Phi_1}, \varphi(x)=\xi\cdot x, \psi(x)=\zeta\cdot x
 %\begin{equation}\label{denotea1}
 \end{equation}
and since $\re l(x)=\xi\cdot x -k(x)$ and $\Im l(x)= \zeta\cdot x +\widetilde{k}(x)$, the above estimates and (\ref{bou}) imply that
\begin{equation}\label{ceroestimate}
\begin{aligned}
& \left \| u_1  \right \|_{H^{1}(\Omega)} = \left \| u_1  \right \|_{L^{2}(\Omega)} + \left \| \nabla u_1  \right \|_{L^{2}(\Omega)}\\
& = \left \| e^{\tau(\varphi + i\psi)} (a_1+r_1) -e^{\tau l}b \right \|_{L^{2}(\Omega)}  \\
&  \qquad+ \left \|  \tau \nabla (\varphi + i\psi)e^{\tau (\varphi + i\psi)} (a_1+r_1)  \right \|_{L^2(\Omega)}\\
&\qquad +  \left \| e^{\tau (\varphi + i\psi)}(\nabla a_1 + \nabla r_1)  -\nabla(e^{\tau l}b)  \right \|_{L^2(\Omega)}\\
& \leq  \left \| e^{\tau(\varphi + i\psi)} (a_1+r_1)  \right \|_{L^{2}(\Omega)}   +  \left \| e^{\tau l}b \right \|_{L^{2}(\Omega)}  \\
&\qquad + \left \| \tau b e^{\tau l}  \nabla l   + e^{\tau l} \nabla b \right \|_{L^{2}(\Omega)}   \\
& \qquad+  \left \| \tau \nabla (\varphi + i\psi)e^{\tau (\varphi + i\psi)} (a_1+r_1) + e^{\tau (\varphi + i\psi)}(\nabla a_1 + \nabla r_1)   \right \|_{L^{2}(\Omega)} \\
& \leq C_1\left \| e^{\tau\varphi} \right \|_{L^\infty(\Omega)}  \left \| a_1+r_1 \right \|_{L^2(\Omega)}  +\left \| e^{\tau l} \right \|_{L^\infty(\Omega)}  \left \| b \right \|_{L^{2}(\Omega)} \\
&\qquad  +\tau \left \| e^{\tau l} \right \|_{L^\infty(\Omega)}  \left \|  b \right \|_{H^{1}(\Omega)} +C_1 \tau \left \| e^{\tau\varphi} \right \|_{L^\infty(\Omega)}  \left \| a_1+r_1 \right \|_{L^2(\Omega)}  \\
&\qquad    +C_1  \left \| e^{\tau\varphi} \right \|_{L^\infty(\Omega)}  \left \| \nabla( a_1+r_1) \right \|_{L^2(\Omega)}    \\
& \leq C_2 \tau e^{\tau c}  \left \|  a_1+r_1\right \|_{H^1(\Omega)}   + C_2\tau e^{\tau c} \left \| b\right \|_{H^1(\Omega)}  \leq C_3 \tau e^{\tau c}. 
\end{aligned}
\end{equation}
We continue in this fashion to compute
\begin{equation}\label{dosestimate}
\begin{aligned}
 \left \| e^{-\tau \varphi}u_1  \right \|_{L^{2}(\partial\Omega)}& =  \left \| e^{-i\tau \psi }(a_1+r_1) + e^{-\tau \varphi}e^{\tau l} b  \right \|_{L^{2}(\partial\Omega)}  \\
 & \leq   \left \| a_1+r_1  \right \|_{L^{2}(\partial\Omega)} +   \left \| e^{-\tau k(x)} b  \right \|_{L^{2}(\partial\Omega)}  \\
 &\leq    \left \|a_1+r_1  \right \|_{H^1(\Omega)} +  \left \|b  \right \|_{H^1(\Omega)} \leq C_4.
 \end{aligned} 
\end{equation}
Finally by denoting $V=a_1+r_1+ e^{-\tau(\varphi+i\psi)} e^{\tau l} b$,  we get
\begin{equation}\label{cuatroestimate}
\begin{aligned}
& \left \| e^{-\tau(\varphi+i\psi)}(\mathcal{L}_{A_1,q_1} - \mathcal{L}_{A_2,q_2}) u_1 \right \|_{L^2(\Omega)} \\
& =  \left \| e^{-\tau(\varphi+i\psi)}  (\mathcal{L}_{A_1,q_1} - \mathcal{L}_{A_2,q_2}) \left[ e^{\tau(\varphi+i\psi)}  V \right] \right \|_{L^2(\Omega)} \\
& \leq   \left \|  (A_1-A_2)\cdot \left[ \tau (\nabla \varphi + i\nabla\psi)V\right]  \right \|_{L^2(\Omega)} \\
& \qquad +  \left \|  (A_1-A_2)\cdot  \nabla V  \right \|_{L^2(\Omega)}\\
& \qquad \qquad  + \left \|  \nabla \cdot (A_1-A_2)V\right \|_{L^2(\Omega)}\\
%&= \left \|  (A_1-A_2)\cdot \left[ \frac{1}{h}(\nabla \varphi + i\nabla\psi) (a_1+r_1+ e^{-\tau(\varphi+i\psi)} e^{\tau l} b) +\nabla (a_1+r_1+ e^{-\tau(\varphi+i\psi)} e^{\tau l} b)\right]  \right \|_{L^2(\Omega)} \\
& \qquad \qquad \qquad + \left \| (A_1^2-A_2^2+q_1-q_2)V \right \|_{L^2(\Omega)} \\
%& = \left \| \right. (A_1-A_2)\cdot \left[ \frac{1}{h}(\nabla \varphi + i\nabla\psi) (a_1+r_1+ e^{-\tau(\varphi+i\psi)} e^{\tau l} b) +\nabla (a_1+r_1+ e^{-\tau(\varphi+i\psi)} e^{\tau l} b)\right] \\
%&\;  \;  \; \; \;  + \nabla \cdot (A_1-A_2) (a_1+r_1+ e^{-\tau(\varphi+i\psi)} e^{\tau l} b) + (A_1^2-A_2^2+q_1-q_2)(a_1+r_1+ e^{-\tau(\varphi+i\psi)} e^{\tau l} b)\left.  \right \|_{L^2(\Omega)}\\
& \leq C_6 \left(   \tau \left \| a_1+r_1+b \right \|_{L^2(\Omega)} + \left \| \nabla(a_1+r_1+b) \right \|_{L^2(\Omega)}\right) \\
&  \leq  C_6\tau  \left \| a_1+r_1+b \right \|_{H^1(\Omega)}\leq C_7\tau.
\end{aligned}
\end{equation}
On the other hand, by Theorem \ref{DSFKSjUs} the function $u_2$ has the form
\[
u_2= e^{-\tau(\xi\cdot x - i\zeta\cdot x)}\left( e^{\Phi_2}g + r_2  \right)
\]
and there exist  two positive constants $C_2$ and $\tau_2$ such that the following estimate 
\begin{equation}\label{estar12}
%\begin{aligned}
%& \left \| e^{\Phi_1 }g \right \|_{H^1(\Omega)}\leq C_1 \left \| g \right \|_{H^1(\Omega))} ,\\
 \left \| \partial^{\alpha} r_2 \right \|_{L^2(\Omega)}\leq C_2 \tau^{\left | \alpha \right |-1} \left \| \overline{g}  \right \|_{H^{2}(\Omega)}, \quad \left | \alpha \right |\leq 1,
%  \left \| \partial^{\alpha} r_2 \right \|_{L^2(\Omega)}\leq C_0 h ^{1- \left | \alpha \right |} \; , \; \left | \alpha \right |\leq 1,
%\end{aligned}
\end{equation}
%\end{equation}
holds true for all $\tau \geq \tau_2$. The above inequality and analogous arguments as used for the boundedness of $u_1$, gives us the following estimates for $u_2$
\begin{equation}\label{unoestimate}
 \left \| u_2  \right \|_{H^{1}(\Omega)}  \leq  C_4\tau e^{\tau c} \left \| \overline{g}  \right \|_{H^{2}(\Omega)}
 \end{equation}
and 
\begin{equation}\label{tresestimate}
 \left \| e^{\tau \xi\cdot x}u_2  \right \|_{L^{2}(\partial\Omega)} \leq C_5\left \| \overline{g}  \right \|_{H^{2}(\Omega)}.
\end{equation}
Thus combining the estimates (\ref{ceroestimate})-(\ref{cuatroestimate}) into (\ref{aleinequality}), and taking into account that  there exists $\tau_3>0$ such that $\tau \leq e^{2\tau c}$, for all $\tau \geq \tau_3$, we get
\begin{align*}
&\left |  \left \langle (\Lambda_1- \Lambda_2)u_1, u_2  \right \rangle_{L^2(\partial\Omega)} \right |\\
& \leq  C\left \| \Lambda_1^{\sharp} - \Lambda_2^{\sharp}  \right \|\left(\tau^2e^{2\tau c} + \tau e^{2\tau c}   \right) \left \| \overline{g} \right \|_{H^2(\Omega)}\\
& \qquad \qquad +C  \left(\tau^{1/2} + 1 \right) \left \| \overline{g} \right \|_{H^2(\Omega)}\\
& \leq C\left(  \tau e^{4\tau c}  \left \| \Lambda_1^{\sharp} - \Lambda_2^{\sharp}  \right \|+ \tau^{1/2} \right) \left \| \overline{g} \right \|_{H^2(\Omega)}.
\end{align*}
We conclude the proof multiplying  by $\tau^{-1}$ both sides of the previous inequality and taking $\tau_0= \max (\tau_1, \tau_2,\tau_3)$.
\end{proof}

Corollary \ref{righthandside} gives us an estimate for the left hand side of the identity (\ref{al}). The task now is to estimate the right hand side multiplied by $\tau^{-1}$, that is to estimate the expression 
\[
\int_{\Omega} \left[ (A_1-A_2)\cdot(\tau^{-1}Du_1 \overline{u}_2 +\tau^{-1} u_1 \overline{Du}_2) + \tau^{-1}(A_1^2-A_2^2+q_1-q_2)u_1\overline{u}_2\right] dx.
\]

For convenience we denote $\rho(x)=(\xi+i\zeta)\cdot x $, $a_1=e^{\Phi_1}$, $u_r=e^{\tau(-\rho+l)}b$ and $a_2=e^{\Phi_2}g$. Hence $u_1$ and $u_2$ have the form (see the Theorems \ref{Zs} and \ref{DSFKSjUs})
\begin{equation}\label{denoted}
u_1= e^{\tau\rho}\left( a_1+r_1-u_r \right)\; , \; u_2= e^{-\tau\overline{\rho}}(a_2+r_2),
\end{equation}
an easy computation shows that

\begin{equation}\label{a4}
\begin{aligned}
 \tau^{-1}Du_1\overline{u}_2&=\left [ e^{\tau\rho} \left( D\rho (a_1+r_1-u_r) + \tau^{-1}D(a_1+r_1-u_r)\right)\right ] \\
 &\qquad \qquad   \times \left [ e^{-\tau\rho} (\overline{a}_2 + \overline{r}_2) \right ]\\
& = D\rho a_1\overline{a}_2 + M_1
\end{aligned}
\end{equation}
and
\begin{equation}\label{a44}
\begin{aligned}
\tau^{-1}u_1\overline{Du}_2&=\left [ e^{\tau\rho}(a_1+r_1-u_r)  \right ] \\
& \qquad\qquad  \times \left [ e^{-\tau\rho} \left(  D\rho(\overline{a}_2+ \overline{r}_2)  +\tau^{-1} \overline{D}(\overline{a}_2+ \overline{r}_2)\right) \right ]  \\
& = D\rho a_1\overline{a}_2 + M_2,
\end{aligned}
\end{equation}
where
\begin{align*}
M_1&= D\rho r_1\overline{a}_2 + \tau^{-1}Da_1(\overline{a}_2+ \overline{r}_2) + \tau^{-1}Dr_1(\overline{a}_2+ \overline{r}_2 )+D\rho( a_1+r_1)\overline{r}_2  \\
&\qquad    -\tau^{-1}e^{-\tau \rho}Du_r( \overline{a}_2+ \overline{r}_2) 
%M_2 &=D\rho a_1 \overline{r}_2+\tau^{-1} a_1\overline{Da}_2+\tau^{-1}a_1\overline{Dr}_2+ D\rho r_1\overline{a}_2 +D\rho r_1\overline{r}_2 +\tau^{-1} r_1 \overline{Da}_2 \\
%&\qquad  + \tau^{-1} r_1\overline{Dr}_2 -u_rD\rho \overline{a}_2-u_rD\rho \overline{r}_2-\tau^{-1}u_r\overline{Da}_2-\tau^{-1}u_r\overline{Dr}_2.
\end{align*}
and
\begin{align*}
M_2 &=D\rho (a_1+r_1) \overline{r}_2+\tau^{-1} a_1(\overline{Da}_2+\overline{Dr}_2)+ D\rho r_1\overline{a}_2  +\tau^{-1} r_1 (\overline{Da}_2 + \overline{Dr}_2   )\\
&-e^{-\tau\rho} u_rD\rho(\overline{a}_2+ \overline{r}_2)+\tau^{-1} e^{-\tau\rho} u_r(\overline{Da}_2+\overline{Dr}_2).
\end{align*}

Now from (\ref{kapa}) we obtain the following estimates
\begin{equation}\label{estar123}
\left \| e^{-\tau\rho}u_r  \right \|_{L^2(\Omega)} \leq C_1 \tau^{-1}, \quad \left \| e^{-\tau\rho}Du_r  \right \|_{L^2(\Omega)} \leq C_1,
\end{equation}
and by a straightforward computation and similar analysis as in the proof of Corollary \ref{righthandside}, there exist two positive constants $C_2$ and $\tau_2$ such that
\begin{equation}\label{a5}
\left \| M_j  \right \|_{L^2(\Omega)} \leq C_2 \tau^{-1} \left \| \overline{g} \right \|_{H^2(\Omega)}, \quad  j=1,2,
\end{equation}
holds true for all $\tau\geq \tau_2$. Thus, Alessandrini's identity,  (\ref{dnmap}) and (\ref{a5}) imply that there exist two positive constants $C_6$ and $\tau_1$ such that the estimate

\begin{equation}
\begin{aligned} 
&2 \int_{\Omega}(A_1-A_2)\cdot D\rho a_1 \overline{a_2} dx \\
&= \tau^{-1} \int_{\Omega} (A_1-A_2)\cdot(Du_1 \overline{u}_2 + u_1 \overline{Du}_2) + (A_1^2-A_2^2+q_1-q_2)u_1\overline{u}_2\\
& \qquad -\int_{\Omega}(A_1-A_2)\cdot (M_1+ M_2) -\tau^{-1} \int_{\Omega}(A_2^2-A_1^2+q_2-q_1)u_1\overline{u_2}\\
& \leq \tau^{-1} \left |    \left \langle (\Lambda_1- \Lambda_2)u_1, u_2  \right \rangle_{L^2(\partial\Omega)}      \right |  + C_3 \left \| M_1+M_2 \right \|_{L^2(\Omega)} \\
&\qquad  + C_4 \tau^{-1} \left \| e^{-\tau \varphi}u_1 \right \|_{L^2(\Omega)}   \left \| e^{\tau \varphi}u_2 \right \|_{L^2(\Omega)}\\
& \leq C_5 \left(  e^{4\tau c}\left \| \Lambda_1^\sharp- \Lambda_2^\sharp \right \|+\tau^{-\frac{1}{2}}+\tau^{-1} + \tau^{-1} \right)  \left \|\overline{g} \right \|_{H^2(\Omega)}    \\
& \leq C_6  \left(  e^{4\tau c}\left \| \Lambda_1^\sharp- \Lambda_2^\sharp \right \|+\tau^{-\frac{1}{2}}  \right)  \left \|\overline{g} \right \|_{H^2(\Omega)},
\end{aligned}
\end{equation}
%\begin{equation}
%\begin{aligned} 
%&2 \int_{\Omega}(A_1-A_2)\cdot D\rho a_1 \overline{a_2} dx \\
%&= \tau^{-1} \int_{\Omega} \left[ (A_1-A_2)\cdot(Du_1 \overline{u}_2 + u_1 \overline{Du}_2) + (A_1^2-A_2^2+q_1-q_2)u_1\overline{u}_2\right] dx\\
%& \qquad -\int_{\Omega}(A_1-A_2)\cdot (M_1+ M_2)dx -\tau^{-1} \int_{\Omega}(A_2^2-A_1^2+q_2-q_1)u_1\overline{u_2}dx\\
%& \leq \tau^{-1} \left |    \left \langle (\Lambda_1- \Lambda_2)u_1, u_2  \right \rangle_{L^2(\partial\Omega)}      \right |  + C_3 \left \| M_1+M_2 \right \|_{L^2(\Omega)} \\
%&\qquad  + C_4 \tau^{-1} \left \| e^{-\tau \varphi}u_1 \right \|_{L^2(\Omega)}   \left \| e^{\tau \varphi}u_2 \right \|_{L^2(\Omega)}\\
%& \leq C_5 \left(  e^{4\tau k}\left \| \Lambda_1^\sharp- \Lambda_2^\sharp \right \|+\tau^{-\frac{1}{2}}+\tau^{-1} + \tau^{-1} \right)  \left \|g \right \|_{H^2(\Omega)}    \\
%& \leq C_6  \left(  e^{4\tau k}\left \| \Lambda_1^\sharp- \Lambda_2^\sharp \right \|+\tau^{-\frac{1}{2}}  \right)  \left \|g \right \|_{H^2(\Omega)},
%\end{aligned}
%\end{equation}
holds true for all $\tau\geq \tau_1$. Hence we have
\begin{equation}\label{mapone}
\begin{aligned}
& \left | (\xi +i\zeta )\cdot \int_{\Omega}(A_1-A_2) e^{\Phi_1 + \overline{\Phi}_2}g dx  \right | \\
& \qquad \qquad \qquad \leq C_6  \left(  e^{4\tau c}\left \| \Lambda_1^\sharp- \Lambda_2^\sharp \right \|+\tau^{-\frac{1}{2}}  \right)  \left \|\overline{g} \right \|_{H^2(\Omega)}
\end{aligned}
\end{equation}
for all $\tau\geq \tau_1$. Next, we use the last inequality to get information on the difference $A_1-A_2$. To do that we will use Lemma \ref{KyUh} in order to remove the function $e^{\Phi_1 + \overline{\Phi}_2}$. Before to state the lemma we have to introduce a new coordinates: every $x\in \mathbb{R}^n$ can be written as follows
\begin{equation}\label{xdf}
x= a \xi +b\zeta+ x^\prime, \quad a=\xi\cdot x, \, b= \zeta\cdot x.
\end{equation}
Thus we consider the coordinates in $\mathbb{R}^n$,  $x \mapsto (a,b,x^\prime)$.
\begin{lem}\label{KyUh}
Let $ \xi, \zeta, \varsigma\in \mathbb{R}^n$ ($n\geq 3$) be orthogonal vectors such that $\left | \xi \right |=\left |\zeta\right |=1$. Consider the coordinates in $\mathbb{R}^n$  given by (\ref{xdf}). If $W\in (L^{\infty}\cap \mathcal{E}^\prime  )(\mathbb{R}^n; \mathbb{C}^n)$ and  $\Phi$ satisfies 
\[
( \xi + i\zeta)\cdot \nabla \Phi +  ( \xi + i\zeta)\cdot W=0
\]
in $\mathbb{R}^n$ then 
\[
( \xi + i\zeta)\cdot \int_{\mathbb{R}^n} W(x)e^{i\varsigma \cdot x} e^{\Phi(x)}g(x)dx= ( \xi + i\zeta)\cdot \int_{\mathbb{R}^n} W(x)e^{i\varsigma \cdot x}g(x) dx,
\]
for all smooth function $g$ depending only on $x^\prime$, that is $g(x)=g(x^\prime)$.  
\end{lem}
\begin{rem}
The proof of this lemma for the case $g\equiv 1$ was given in \cite{KU}. See also Lemma $2.6$ in \cite{Tz}. The proof for any $g$ depending only on $x^\prime$ is similar to the proof of Proposition $3.3$ in \cite{KU}.
\end{rem}
%\begin{proof}
%See Proposition $3.3$ in \cite{KU} and also Lemma $2.6$ in \cite{Tz}.
%\end{proof}

\begin{prop}\label{RTRA}
Let $\Omega\subset \mathbb{R}^n$ be a bounded open set with smooth boundary. Let $\xi \in N\subset S^{n-1}$ and $\zeta \in S^{n-1}$ such that $ \xi \cdot \zeta=0$. Let $M>0$ and consider $A_1, A_2\in \mathscr{A}(\Omega, M)$ and $q_1, q_2\in L^\infty(\Omega)$. If $A_1=A_2$ on $\partial\Omega$ then there exist two positive constants $\tau_0$ and $C>0$ (both depending on $n, \Omega$, M) such that
\begin{equation}\label{Remove}
\left |\mu\cdot \int_{\Omega}(A_1-A_2)gdx  \right |\leq C \left | \mu \right |  \left |\log \left \| \Lambda_1^{\sharp}- \Lambda_2^{\sharp} \right \|  \right |^{-1/2} \left \| \overline{g} \right \|_{H^2(\Omega)}
\end{equation}
holds true for all $\mu \in span\left \{\xi, \zeta\right \}$, provided that  $ \left \|  \Lambda_1^{\sharp}-\Lambda_2^{\sharp} \right \| \leq e^{-4c\tau_0}$.
\end{prop}
\begin{proof}
We first prove the Proposition for $\mu=\xi +i\zeta$. The equations (\ref{Phiuno}), (\ref{Phidos}) and  (\ref{guno1}) imply that 
\[
(\xi +i\zeta)\cdot \nabla (\Phi_1+\overline{\Phi}_2) +i (\xi +i\zeta)\cdot \left [ \chi_{\Omega}(A_1-A_2) \right ]=0
\]
in $\Omega$. Notice that the above equation could be extended to all $\mathbb{R}^n$ by considering $A_1-A_2=0$ on $\mathbb{R}^n\setminus \Omega$, since $A_1=A_2$ on $\partial\Omega$. Then applying Lemma \ref{KyUh} with $\varsigma=0$, $W=i\chi_{\Omega}(A_1-A_2)$, $\Phi=\Phi_1+\overline{\Phi}_2 $ and a function $g$ depending only on $x^\prime$(notice that such function $g$ satisfies (\ref{guno1})), we obtain
\begin{equation}\label{L}
\begin{aligned}
&  ( \xi + i\zeta)\cdot \int_{\Omega} (A_1-A_2)ge^{\Phi_1+\overline{\Phi}_2}dx\\
&\qquad = ( \xi + i\zeta)\cdot \int_{\mathbb{R}^n} \chi_{\Omega}(A_1-A_2)ge^{\Phi_1+\overline{\Phi}_2}dx\\
&\qquad =  ( \xi + i\zeta)\cdot \int_{\mathbb{R}^n}\chi_{\Omega}(A_1-A_2)g dx\\
&\qquad =  ( \xi + i\zeta)\cdot \int_{\Omega}(A_1-A_2)g dx.
\end{aligned}
\end{equation}
On the other hand, there exists $\tau_2>0$ such that 
\begin{equation}\label{sucesion}
e^{-2\tau c}\leq \tau^{-1/2},
\end{equation}
for all $\tau \geq \tau_2$. Let $\tau_1>0$ be such that (\ref{mapone}) is satisfied. Taking $\tau_0=\max(\tau_1, \tau_2)$, it is easy to check that
\[
\tau:=   \dfrac{1}{8} c^{-1} \left | \log \left \| \Lambda_1^\sharp- \Lambda_2^\sharp \right \| \right | \geq \tau_0 ,
\]
whenever
\[
 \left \|  \Lambda_1^{\sharp}-\Lambda_2^{\sharp} \right \| \leq e^{-4c\tau_0}.
\]
Thus, from (\ref{L}) and replacing the above inequalities into (\ref{mapone}), we get
\begin{equation}\label{xipluszeta}
\left |(\xi + i\zeta)\cdot \int_{\Omega}(A_1-A_2)gdx  \right |\leq C_1  \left |\log \left \| \Lambda_1^{\sharp}- \Lambda_2^{\sharp} \right \|  \right |^{-1/2} \left \| \overline{g} \right \|_{H^2(\Omega)}.
\end{equation}
By Remark \ref{minuszeta}, we can apply the previous arguments again, with $(\xi + i\zeta)$ replaced by $(\xi - i\zeta)$, to obtain
\begin{equation}\label{ximinuszeta}
\left |(\xi - i\zeta)\cdot \int_{\Omega}(A_1-A_2)gdx  \right |\leq C_2  \left |\log \left \| \Lambda_1^{\sharp}- \Lambda_2^{\sharp} \right \|  \right |^{-1/2} \left \| \overline{g} \right \|_{H^2(\Omega)}.
\end{equation}
Combining (\ref{xipluszeta}) and (\ref{ximinuszeta}) we conclude the proof.

\end{proof}

%%%%%%%%%%%%%%%%%%%%%%%%%%%%%%%%%%%%%%%%%%%%%%%%%%%%%%%%%%%%%%%%%%%%%%%%%%%%%%%%%%%%%%%%%%%%%%%%%%%%%%%%%%%%

\subsection{Radon transform and its applications}
Let $f$ be a function on $\mathbb{R}^n$, integrable on each hyperplane in $\mathbb{R}^n$. These hyperplanes can be parametrized by its unit normal vector and distance to the origen: $\theta$ and $s$, respectively. Thus we set
\[
H(s, \theta)= \left \{ x\in \mathbb{R}^n \; : \;  \left \langle x, \theta  \right \rangle =s \right \}
\]
and in this setting the Radon transform of $f$ is defined by
\[
(\pmb{R}f)(s,\theta )=\int_H f(x)d\mu_H= \underset{\theta^{\perp}}{\int} f(s\theta +y) dy,
\]
whenever the integral exists. Here $\theta^{\perp}$ denotes the set of orthogonal vectors to $\theta$. This is the definition of the Radon transform with respect to the origin, but later we will have to know this transform at some arbitrary point in $\mathbb{R}^n$. In this case the natural definition is as follows. For $y_0\in \mathbb{R}^n$,  we set
\[
H_{y_0}= \left \{ x \in \mathbb{R}^n \; : \;  \left \langle x-y_0, \theta  \right \rangle =s \right \}
\]
for some $\theta\in S^{n-1}$ and $s\in \mathbb{R}$. With respect to these parameters we define 
\[
\pmb{R}_{y_0}f (s,\theta)= \int_{H_{y_0}} f d\mu_{H_{y_0}}.
\]

It is easy to check that for all $y_0\in \mathbb{R}^n$, $\theta\in S^{n-1}$ and $s\in \mathbb{R}$, we have the following relation
\begin{equation}
\pmb{R}_{y_0}f (s,\theta) = (\pmb{R}f)(s + \left \langle y_0, \theta  \right \rangle,\theta ).
\end{equation}

We now define the Fourier transform with respect to the first  variable of a function $F:\mathbb{R}\times S^{n-1} \rightarrow \mathbb{R}$ by

\[
\widehat{F}( \sigma, \theta)= (2 \pi)^{-1/2} \underset{\mathbb{R}}{\int} e^{-is\sigma}F( s, \theta)ds.
\]

For $\alpha\geq 0$ we define the Sobolev space $H^{\alpha}( \mathbb{R}\times S^{n-1}) $ as the subspace of $L^2(\mathbb{R}\times S^{n-1})$ with the norm

\[
 \left \| F \right \|_{H^\alpha(\mathbb{R}\times S^{n-1})} = \left( \underset{S^{n-1}}{\int} \underset{\mathbb{R}}{\int} (1+\sigma^2)^\alpha  \left | \widehat{F}(\sigma, \theta) \right |^2 d\sigma d\theta     \right)^{1/2}.
\]

The following two results can be found in \cite{Na}:  for each $\alpha\geq 0$ there exist positive constants $c$ and $C_0$, both depending on $\alpha$ and $n$, such that

\begin{equation}\label{Rfequation}
 \left \| \pmb Rf \right \|_{H^{\alpha + (n-1)/2}(\mathbb{R}\times S^{n-1})} \leq C_0  \left \| f \right \|_{H^{\alpha}(\mathbb{R}^n)},
\end{equation}
whenever $f$ has a compact support. Moreover for all $f\in H^1(\mathbb{R}^n)$ with compact support, the following identity holds in the sense of the distributions in $C^{\infty}_0(\mathbb{R})$
\begin{equation}\label{properties}
\theta_i\dfrac{\partial}{\partial s}(\pmb{R}f)(\cdot, \theta)= \pmb{R}(\partial_{x_i}f)(\cdot, \theta),
\end{equation}
for any $\theta\in S^{n-1}$ and $i=1,2,\ldots,n$. Here $\theta_i$ denotes the $i$-th coordinate of $\theta$.
The next result was proved by Caro, Dos Santos Ferreira and Ruiz, see Theorem $2.5$ in \cite{CDSFR}. This gives a stability estimate for the Radon transform in a suitable space and will be the main tool to improving our stability result for both magnetic an electrical potentials. Before stating their result we introduce the set $X$ as the subspace of $L^{1}(\mathbb{R}^n)$ with the norm
\[
\left \| F \right \|_X = \int_{\mathbb{R}} (1+\left | s \right | )^n   \left \| \pmb{R}F(s, \cdot) \right \|_{L^1(S^{n-1})}ds.
\]
and recall the distance on the sphere: $d_{S^{n-1}}(x,y)=\arccos(\left \langle x,y \right \rangle)$.
\begin{thm}\label{cdsr}
Let $M\geq 1, \alpha>0$ and $\beta\in \left ( 0,1 \right )$. Given $y_0\in \mathbb{R}^n$ and  $\theta_0\in S^{n-1}$,  consider the set
\[
\Gamma= \left \{\theta \in S^{n-1}\; : \; d_{S^{n-1}}(\theta_0,\theta)< \arcsin \beta  \right \}
\]
and the domain of dependence of the Radon transform by
\[
E= \left \{ x\in \mathbb{R}^n \; : \; \left \langle  \theta, x-y_0 \right \rangle=s \; , \; s\in \left ( -\alpha, \alpha \right ) \; , \; \theta\in \Gamma  \right \}.
\]
Assume that there exist two constants $p$, with $1\leq p<\infty$ and $\lambda$, with $0<\lambda<p^{-1}$;  such that a function $F$ satisfies the following conditions:
\begin{item}
\item[(a).] $\chi_{E}F\in X \cap L^\infty(\mathbb{R}^n)$, where $\chi_E$ denotes the characteristic function of the set $E$. Moreover
\[
\left \| F \right \|_{L^\infty(E)} + \left \| \chi_E F \right \|_X \leq M.
\]
\item[(b).] $y_0\in \supp F$ and $\supp F \subset \left \{ x\in \mathbb{R}^n \; : \; \left \langle x-y_0, \theta_0 \right \rangle \leq 0 \right \} $. 
\item[(c).] The function $F$ satisfies the following $(\lambda, p)$-Besov regularity
\[
\int_{\mathbb{R}^n} \dfrac{\left \| \chi_E F(\cdot)-(\chi_E F)(\cdot-y)  \right \|_{L^p(\mathbb{R}^n)}^p}{\left | y \right |^{n+\lambda p}}dy\leq M^p.
\] 
\end{item}
Then there exists a positive constant $C$ (depending on $G, M, \alpha, \beta, \lambda$), such that
\[
\left \|F  \right \|_{L^p(G)} \leq C    \left |  \log \int_{-\alpha}^{\alpha}  (1+ \left | s \right |)^n \left \| \pmb{R}_{y_0} F(s, \cdot) \right \|_{L^1(\Gamma)} ds   \right |^{-\lambda/2},
\]
where
\begin{equation}\label{gii}
G= \left \{  x\in \mathbb{R}^n \; : \;  \left | x-y_0 \right |< \dfrac{\alpha}{8 \cosh(8\pi/\beta)}  \right \}.
\end{equation}
\end{thm}
\begin{rem}\label{albe}
In our context, the constant $\beta$ stands for size of the set $N\subset S^{n-1}$ and $\left(-\alpha, \alpha \right)$ is the interval where we have control of the Radon transform $\pmb{R}F(\cdot, \theta)$, with $\theta\in S^{n-1}$ and $F\in X$. Notice that for fixed $y_0\in \mathbb{R}^n$ and  $\beta>0$ we can take $\alpha$ large enough so that $\Omega\subset G$. We will use this facts in the proof of Theorem \ref{SMP}.
\end{rem}
%%%%%%%%%%%%%%%%%%%%%%%%%%%%%%%%%%%%%%%%%%%%%%%%%%%%%%%%%%%%%%%%%%%%%%%%%%%%%%%%%%%%%%%%%%%%%%%%%%%%%%%%%%%%

\subsection{Proof of Theorem \ref{SMP}}
We start by rewriting the estimate from Proposition \ref{RTRA} in the natural coordinates of the Radon transform of $\chi_\Omega(A_1-A_2)$. More precisely 
\begin{cor} If we consider the open set in $S^{n-1}$
\begin{equation}\label{M}
M= \underset{\xi \in N}{\bigcup}\left [ \xi \right ]^{\perp},
\end{equation}
then for any $\widetilde{g}\in C^\infty(\mathbb{R})$ there exist two positive constants $C$ and $\tau_0$ (both depending on $n,\Omega$ and the a priori bounds of $\left \| A_j \right \|_{C^2(\overline{\Omega})}$ and $\left \| q_j \right \|_{L^\infty(\Omega)}$) such that the following estimate 
\begin{equation} \label{muuu3}
\begin{aligned}
&\left |  \mu \cdot \int_{\mathbb{R}}  \widetilde{g}(s) (\pmb{R}\left [  \chi_{\Omega}(A_1-A_2)\right ] )(s, \theta )ds \right |\\
&\qquad \qquad  \leq  C \left |  \mu \right | \left |\log \left \| \Lambda_1^{\sharp}- \Lambda_2^{\sharp} \right \|  \right |^{-1/2} \left \| \widetilde{g} \right \|_{H^2(\mathbb{R})},
\end{aligned}
\end{equation}
holds true for all $\theta \in M$ and $\mu \in \theta^\perp$.
\end{cor}
\begin{proof}
The main idea of the proof  is to see the left hand side of (\ref{Remove}) as the Radon transform of a suitable function. So consider $\xi\in N\subset S^{n-1}$ and $\zeta\in S^{n-1}$ such that $\xi\cdot \zeta=0$. We take some  $\theta \in \left [ \xi,\zeta \right ]^{\perp}$ with $\left | \theta \right |=1$. Thus, every $x\in \mathbb{R}^n$ can be written as
\[
x=t\xi + r\zeta + s\theta + x^\prime, \quad x^{\prime} \in \left [ \xi,\zeta, \theta \right ]^{\perp}.
\]
This decomposition can be done since $n\geq 3$. Now we consider the change of coordinates in $\mathbb{R}^n$ defined by $\Psi: x \mapsto (t,r,s, x^{\prime})$; and a straightforward computation shows that  if $g\in C^{\infty}(\mathbb{R}^n)$ satisfies $(\xi + i\zeta)\cdot \nabla g=0$, then the function $\widetilde{g}:= \overline{g}\circ \Psi^{-1} $ satisfies
\begin{equation}\label{wideg}
(\partial_t - i\partial_r)\widetilde{g}=0,
\end{equation}
where $\partial_t$ and $\partial_r$ denote the partial derivative with respect to $t$ and $r$, respectively. Notice that any function $\widetilde{g}:= \widetilde{g}(s)$  that depend only on the variable $s$, satisfies (\ref{wideg}). For $\Psi$-coordinates we have $dx=dx^{\prime} dtdrds$ and  for every  $\mu \in \left [  \xi, \zeta \right ]$, we obtain
\begin{align*}
&\mu\cdot \int_{\Omega}(A_1-A_2)gdx= \mu\cdot \int_{\mathbb{R}^n}\chi_{\Omega}(A_1-A_2)g dx\\
&= \mu\cdot \int_{\mathbb{R}^3}\int_{\left [ \xi,\zeta, \theta \right ]^{\perp}}\left [  \chi_{\Omega}(A_1-A_2)\circ \Psi^{-1}\right ] \left[ g\circ \Psi^{-1} dx^\prime \right]dtdrds\\
& = \mu\cdot \int_{\mathbb{R}} \widetilde{g}(s) \left( \int_{\mathbb{R}^2}\int_{\left [ \xi,\zeta, \theta \right ]^{\perp}}\left [  \chi_{\Omega}(A_1-A_2)\right ] (t\xi + r\zeta + s\theta + x^\prime)  dx^\prime dtdr \right)  ds\\
&  = \mu\cdot \int_{\mathbb{R}} \widetilde{g}(s) \left( \int_{\theta^{\perp}}\left [  \chi_{\Omega}(A_1-A_2)\right ] (s\theta + y)  dy \right)  ds\\
& = \mu \cdot \int_{\mathbb{R}}  \widetilde{g}(s) (\pmb{R}\left [  \chi_{\Omega}(A_1-A_2)\right ] )(s, \theta )ds.
\end{align*}
This equality and estimate (\ref{Remove}) imply (\ref{muuu3}). 
\end{proof}
In particular estimate (\ref{muuu3}) holds for the vectors $\mu_{ij}= \theta_i e_j - \theta_j e_i$ with $i,j=1,2, \ldots,n$. Here $(e_i)_{i=1}^{n}$ denotes the canonical basis of $\mathbb{R}^n$ and $\theta_i$ the $i$-th component of $\theta$.  Denoting $\widetilde{A}= \chi_{\Omega}(A_1-A_2)$, it follows that $\widetilde{A}$ belongs to $H^{1}(\mathbb{R}^n)$ and has a compact support. Thus from (\ref{properties}), for all $\widetilde{h}\in C^\infty_0(\mathbb{R})$ and all $i,j= 1,2, \ldots,n$; we get 
\begin{align*}
& \mu_{i,j} \cdot \int_{\mathbb{R}}  \dfrac{\partial}{\partial s} \widetilde{h}(s) (\pmb{R}\left [  \chi_{\Omega}(A_1-A_2)\right ] )(s, \theta )ds\\
& = \int_{\mathbb{R}}   \dfrac{\partial}{\partial s}\widetilde{h}(s) \left [  \theta_i e_j - \theta_j e_i\right ] \cdot (\pmb{R} \widetilde{A})(s, \theta )ds\\
&  = \int_{\mathbb{R}}  \dfrac{\partial}{\partial s}\widetilde{h}(s) \left [  \theta_i \left( \pmb{R}\widetilde{A}_j\right)(s, \theta ) - \theta_j \left( \pmb{R}\widetilde{A}_i\right)(s, \theta )\right ] ds\\
& = -\int_{\mathbb{R}} \widetilde{h}(s) \left [  \theta_i  \dfrac{\partial}{\partial s}\left( \pmb{R}\widetilde{A}_j\right)(s, \theta ) - \theta_j  \dfrac{\partial}{\partial s} \left( \pmb{R}\widetilde{A}_i\right)(s, \theta )\right ] ds\\
& = -\int_{\mathbb{R}} \widetilde{h}(s) \left [  \pmb{R}\left( \partial_{x_i}\widetilde{A}_j-\partial_{x_j}\widetilde{A}_i \right)   \right ] (s,\theta)ds.
\end{align*}
From this and (\ref{muuu3}) it follows that for all $\theta\in M$ we have
\begin{align*}
&\left | \int_{\mathbb{R}} \widetilde{h}(s) \left [  \pmb{R}\left( \partial_{x_i}\widetilde{A}_j-\partial_{x_j}\widetilde{A}_i \right)   \right ] (s,\theta)ds \right |\\
& \qquad = \left |  \mu_{i,j} \cdot \int_{\mathbb{R}}  \dfrac{\partial}{\partial s} \widetilde{h}(s) (\pmb{R}\left [  \chi_{\Omega}(A_1-A_2)\right ] )(s, \theta )ds \right |\\
&\qquad  \leq  C\left |\mu_{i,j}  \right | \left |\log \left \| \Lambda_1^{\sharp}- \Lambda_2^{\sharp} \right \|  \right |^{-1/2} \left \| \partial_s \widetilde{h} \right \|_{H^2(\mathbb{R})}\\
&\qquad  \leq  C \left |\log \left \| \Lambda_1^{\sharp}- \Lambda_2^{\sharp} \right \|  \right |^{-1/2} \left \| \widetilde{h} \right \|_{H^3(\mathbb{R})},
\end{align*}
which implies that
\[
\left \| \pmb{R}\left( \partial_{x_i}\widetilde{A}_j-\partial_{x_j}\widetilde{A}_i \right)  \right \|_{H^{-3}(\mathbb{R}; L^{\infty}(M))}  \leq  C \left |\log \left \| \Lambda_1^{\sharp}- \Lambda_2^{\sharp} \right \|  \right |^{-1/2}.
\]
On the other hand, from (\ref{Rfequation}) we obtain
\[
 \left \| \pmb{R}\left( \partial_{x_i}\widetilde{A}_j-\partial_{x_j}\widetilde{A}_i \right)  \right \|_{H^{\frac{n-1}{2}}(\mathbb{R}; L^2(M))} \leq C_1  \left \|  \partial_{x_i}\widetilde{A}_j-\partial_{x_j}\widetilde{A}_i\right \|_{L^{2}(\mathbb{R}^n)}\leq C_2.
\]
Thus, by standard interpolation between the spaces $H^{-3}(\mathbb{R}; L^{\infty}(M))$ and $H^{\frac{n-1}{2}}(\mathbb{R}; L^2(M))$, we have
\begin{equation}\label{interpol}
\begin{aligned}
&  \left \| \pmb{R}\left( \partial_{x_i}\widetilde{A}_j-\partial_{x_j}\widetilde{A}_i \right)  \right \|_{L^2(\mathbb{R}; L^{{(n+5)}/{3}}(M))}\\
 & \qquad \qquad \qquad \leq C_3  \left |\log \left \| \Lambda_1^{\sharp}- \Lambda_2^{\sharp} \right \|  \right |^{-\frac{1}{2}{(n-1)}/{(n+5)}}.
\end{aligned}
\end{equation}
\\

The next step will be to verify the three conditions of Theorem \ref{cdsr} for the function $F_{i,j}:=\partial_{x_i}\widetilde{A}_j- \partial_{x_j}\widetilde{A}_i$, for fixed $i\neq  j$; $i,j \in \left \{1,2, \ldots,n\right \}$.  Let us start with the supporting condition $(b)$. Indeed, take $\theta_0\in M$ and by translation, there exists $y_0\in \supp F_{i,j}$ such that
\[
\supp F_{i,j}\subset  \left \{ x\in \mathbb{R}^n \; : \; \left \langle x-y_0, \theta_0 \right \rangle \leq 0 \right \}. 
\]
This can be done because $\Omega$ is a bounded open set. Since $M$ is a open neighborhood of $\theta_0$ and from estimate (\ref{interpol}), we can control the Radon transform of $F_{i,j}$ for $s\in \mathbb{R}$ and $\theta\in M$. Thus, from Remark (\ref{albe}), there exists $\beta\in\left(0,1\right)$ such that the condition $(a)$ is satisfied for any $\alpha>0$. Moreover, by taking $\alpha$ large enough it follow that $\supp F_{i,j}\subset \overline{\Omega} \subset G$, where $G$ is defined by (\ref{gii}). The condition $(c)$ is satisfied for $p=2$ and $0<\lambda<1/2$. Thus, Theorem \ref{cdsr} ensures that there exists $C>0$ such that
\begin{equation}\label{a123}
\left \|F_{i,j}  \right \|_{L^2(\mathbb{R}^n)} \leq C    \left |  \log \int_{-\alpha}^{\alpha}  (1+ \left | s \right |)^n \left \| \pmb{R}_{y_0} F_{i,j}(s, \cdot) \right \|_{L^1(\Gamma)} ds   \right |^{-\lambda/2}.
\end{equation}
Here the set $\Gamma$ is where we have the control of the Radon transform on the $\theta$-variable. In our case (see the estimate (\ref{interpol})) we have the control on $M$. Now we set
\[
L= \underset{\theta\in M}{\sup}  \left \| (1+ \left | \cdot- \left \langle \theta, y_0 \right  \rangle \right |)^n  \right \|_{L^2(\left | s \right |\leq \alpha +\left | y_0 \right | )}
\]
and denote by $\left | M \right |$ the measure of $M$. Then the inequality (\ref{a123}), (\ref{interpol}), Fubini's theorem and H\"older's inequality applied twice, imply that
\begin{align*}
&\int_{-\alpha}^{\alpha}  (1+ \left | s \right |)^n \left \| \pmb{R}_{y_0} F_{i,j}(s, \cdot) \right \|_{L^1(M)} ds\\
&=  \int_{-\alpha}^{\alpha}  (1+ \left | s \right |)^n  \int_M  \left | (\pmb{R} F_{i,j} )(s+ \left \langle \theta, y_0 \right  \rangle, \theta)   \right | d\theta ds \\
& \leq\int_M \int_{-(\alpha +\left | y_0 \right | )}^{\alpha +\left | y_0 \right | }    (1+ \left | s- \left \langle \theta, y_0 \right  \rangle \right |)^n  \left | (\pmb{R} F_{i,j} )(s, \theta)   \right | ds  d\theta  \\
& \leq \int_M   \left \| (1+ \left | \cdot- \left \langle \theta, y_0 \right  \rangle \right |)^n  \right \|_{L^2(\left | s \right |\leq \alpha +\left | y_0 \right | )}    \left \| (\pmb{R} F_{i,j} )(\cdot, \theta)  \right \|_{L^2(\left | s \right |\leq \alpha +\left | y_0 \right | )} d\theta \\
&\leq  L \int_M \left(  \int_{\mathbb{R}}  \left | \pmb{R}F_{i,j}(s,\theta) \right | ^2 ds \right)^{1/2}d\theta  \\
&  \leq L \left | M \right |^{\frac{n+2}{n+5}}  \left(   \int_M   \left(  \int_{\mathbb{R}} \left | \left( \pmb{R}F_{i,j}\right) (s, \theta)  \right |^2       \right)^{(n+5)/6}   d\theta \right)^{3/(n+5)}\\
& =  L \left | M \right |^{\frac{n+2}{n+5}}  \left \| \pmb{R}\left( \partial_{x_i}\widetilde{A}_j-\partial_{x_j}\widetilde{A}_i \right)  \right \|_{L^2(\mathbb{R}; L^{{(n+5)}/{3}}(M))}\\
& \leq C_4  \left |\log \left \| \Lambda_1^{\sharp}- \Lambda_2^{\sharp} \right \|  \right |^{-\frac{1}{2}{(n-1)}/{(n+5)}}.
\end{align*}
We conclude the proof by taking logarithm  to both sides of the above inequality and taking into account estimate (\ref{a123}).

%\begin{align*}
%&\int_{-\alpha}^{\alpha} \int_M \left( 1+  \left |s- \left \langle \theta, y_0 \right \rangle  \right | ^n\right)\left | \left( \pmb{R}_{y_0}  F_{i,j}\right) (s, \theta)\right |ds d\theta\\
%& \leq C_1 \left(   \int_M   \left(  \int_{-\alpha}^{\alpha} \left | \left( \pmb{R}F_{i,j}\right) (s, \theta)  \right |^2       \right)^{(n+3)/4}   d\theta \right)^{2/(n+3)}\\
%&  \leq C_1 \left(   \int_M   \left(  \int_{\mathbb{R}} \left | \left( \pmb{R}F_{i,j}\right) (s, \theta)  \right |^2       \right)^{(n+5)/6}   d\theta \right)^{3/(n+5)}\\
%& = C_1 \left \| \pmb{R}\left( \partial_{x_i}\widetilde{A}_j-\partial_{x_j}\widetilde{A}_i \right)  \right \|_{L^2(\mathbb{R}; L^{{(n+5)}/{3}}(M))}\\
%& \leq C_4  \left |\log \left \| \Lambda_1^{\sharp}- \Lambda_2^{\sharp} \right \|  \right |^{-\frac{1}{2}{(n-1)}/{(n+5)}}.
%\end{align*}

\section{Stability estimate for the electrical potential}

The goal of this section is to prove Theorem \ref{SEP}. The idea will be to combine the gauge invariance for the DN map and the stability result already proved for the magnetic fields. This kind of arguments involve a  Hodge decomposition as in Tzou, see \cite{Tz}. We recall this decomposition in the following lemma.

\begin{lem}\label{hd}
Let $\Omega \subset \mathbb{R}^n$ be a simply-connected open bounded set with connected smooth boundary. If $A_1, A_2\in W^{2,p}(\Omega)$ with $p\geq 2$, and $A_1=A_2$ on $\partial \Omega$. Then there exist a constant $C>0$ and $\omega\in W^{3,p}(\Omega)\cap H_{0}^1(\Omega)$ such that 
\[
\left \| A_1-A_2-d\omega \right \|_{W^{1,p}(\Omega)} \leq C \left \|d(A_1-A_2)  \right \|_{L^p(\Omega)}
\]
and
\[
\left \| \omega \right \|_{W^{3,p}(\Omega)} \leq C \left \| A_1-A_2 \right \|_{W^{2,p}(\Omega)}.
\]
\end{lem}

From now on we consider the bounded open set $\Omega$ to be simply-connected with connected smooth boundary. Let $A_1, A_2\in W^{2, \infty}(\Omega)$; $q_1, q_2\in L^{\infty}(\Omega)$ and $p>n$. Then, by Morrey's inequality and Lemma \ref{hd} there exist a constant $C>0$ and  $w\in W^{3,p}(\Omega)\cap H_{0}^1(\Omega)$ such that
\begin{equation}\label{ggg1}
\left \| A_1-A_2-\nabla \omega \right \|_{C^{0, 1-\frac{n}{p}} (\overline{\Omega})} \leq C \left  \|d(A_1-A_2)  \right \|_{L^p(\Omega)}
\end{equation}
and
\begin{equation}\label{ggg2}
\left \| \omega \right \|_{L^{\infty}(\Omega)} + \left \| \nabla \omega \right \|_{L^{\infty}(\Omega)}+ \left \| \Delta \omega \right \|_{L^{\infty}(\Omega)} \leq C \left \| A_1-A_2 \right \|_{W^{2,p}(\Omega)}.
\end{equation}

We denote by $\widetilde{A}_1= A_1 -\nabla \omega / 2$ and $\widetilde{A}_2= A_2 +\nabla \omega / 2$. Thus, by Lemma $3.1$ in \cite{KU}, we have the identities
\begin{equation}\label{ginv1}
e^{i\omega/2}\mathcal{L}_{\widetilde{A}_1,q_1}e^{-i\omega/2}= \mathcal{L}_{\widetilde{A_1}, q_1} \; , \; \Lambda_{A_1,q_1}= \Lambda_{\widetilde{A}_1,q_1}
\end{equation}
and
\begin{equation}\label{ginv2}
e^{-i\omega/2}\mathcal{L}_{\widetilde{A}_1,q_1}e^{i\omega/2}= \mathcal{L}_{\widetilde{A_2}, \overline{q}_2} \; , \; \Lambda_{A_2,\overline{q}_2}= \Lambda_{\widetilde{A}_2,\overline{q}_2}.
\end{equation}
\\

In Section $1$ we used identity (\ref{al}) to isolate $A_1-A_2$ and then using CGO solutions we obtain the estimate from Corollary \ref{righthandside}. Now we follow the same ideas. We use again Alessadrini's identity in order to isolate $q_1-q_2$ and we obtain stability result for electrical potentials by using similar estimates as in Proposition  \ref{leftsideAI} and Corollary \ref{righthandside}. We start by denoting $\widetilde{\Lambda}_i = \Lambda_{\widetilde{A}_i, q_i}$ for $i=1,2$. If $U_1, U_2\in H^1(\Omega)$ such that $\mathcal{L}_{\widetilde{A}_1, q_1}U_1=0$ and $\mathcal{L}_{\widetilde{A}_2, \overline{q_2}}U_2=0$ then by identity (\ref{al}) we have
\begin{equation}\label{algee}
\begin{aligned}
& \left \langle  (\widetilde{\Lambda}_1 - \widetilde{\Lambda}_2)U_1,U_2 \right \rangle_{L^2(\partial\Omega)}\\
&=\int_{\Omega} \left[ (\widetilde{A}_1-\widetilde{A}_2)\cdot(DU_1 \overline{U}_2 + U_1 \overline{DU}_2) + (\widetilde{A}_1^2-\widetilde{A}_2^2+q_1-q_2)U_1\overline{U}_2\right].
 \end{aligned}
\end{equation}

\begin{prop}\label{qq1} Let $\Omega\subset\mathbb{R}^n$ be a bounded open set with smooth boundary. Consider two positive constants $M$ and $\sigma$. Let $A_1, A_2\in \mathscr{A}(\Omega, M)$ with $A_1=A_2$ on $\partial\Omega$; and $q_1, q_2\in  \mathscr{Q}(\Omega, M, \sigma)$. If $U_1, U_2\in H^1(\Omega)$ satisfie $\mathcal{L}_{\widetilde{A}_1, q_1}U_1=0$ and $\mathcal{L}_{\widetilde{A}_2, \overline{q}_2}U_2=0$, then there exist two positive constants $\tau_0$ and $C$ (both depending on $n, \Omega, M, \sigma$) such that  the estimate
\begin{equation}\label{alei}
\begin{aligned}
&\left |  \left \langle (\widetilde{\Lambda}_1- \widetilde{\Lambda}_2)U_1, U_2  \right \rangle_{L^2(\partial\Omega)} \right |\\
& \leq  C\left \| \widetilde{\Lambda}_1^{\sharp} - \widetilde{\Lambda}_2^{\sharp}  \right \|  \left(  \left \| U_1 \right \|_{H^1(\Omega)}  \left \| U_2 \right \|_{H^1(\Omega)}  + e^{\tau c}\left \| U_1 \right \|_{H^1(\Omega)}   \left \| e^{\tau \xi \cdot x}U_2 \right \|_{L^{2}(\partial\Omega)} \right)  \\
&  \qquad + C\tau^{-\frac{1}{2}} \left \| e^{-\tau \xi \cdot x}(\mathcal{L}_{\widetilde{A}_1,q_1} - \mathcal{L}_{\widetilde{A}_2,q_2}) U_1 \right \|_{L^2(\Omega)}  \left \| e^{\tau \xi \cdot x} U_2 \right \|_{L^{2}(\partial\Omega)} \\
& \qquad + C\left \| \widetilde{A}_1-\widetilde{A}_2 \right \|_{L^\infty(\Omega)} \left \| e^{-\tau \xi \cdot x}U_1 \right \|_{L^{2}(\partial\Omega)} \left \| e^{\tau \xi \cdot x} U_2 \right \|_{L^{2}(\partial\Omega)}
\end{aligned}
\end{equation}
holds true  for all $\tau \geq \tau_0$ and all $\xi\in N$.
\end{prop}
\begin{proof}
The proof is similar to the proof of Proposition \ref{leftsideAI}, with $A_i$ replaced by $\widetilde{A}_i$ for $i=1,2$.  We give the proof only for completeness and we will take extra care when the term $\widetilde{A}_1- \widetilde{A}_2= A_1-A_2-\omega$ appears in the following estimates. Throughout this proof we take into account the notation from Proposition  \ref{leftsideAI}. Let us begin with the following identity
\begin{equation}\label{decomposition1}
\begin{aligned}
&\left \langle (\widetilde{\Lambda}_1- \widetilde{\Lambda}_2)U_1, U_2  \right \rangle_{L^2(\partial\Omega)} = \left \langle\chi (\widetilde{\Lambda}_1- \widetilde{\Lambda}_2)U_1, U_2  \right \rangle_{L^2(\partial\Omega)}   \\
&\qquad \qquad \qquad \qquad \qquad \qquad \;  \; +  \left \langle(1-\chi) (\widetilde{\Lambda}_1- \widetilde{\Lambda}_2)U_1, U_2  \right \rangle_{L^2(\partial \Omega)}.
\end{aligned}
\end{equation}
We estimate the first term of the right hand side in the above identity as follows 
\begin{equation}\label{unoxx12}
\left |  \int_{\partial\Omega}\chi   (\widetilde{\Lambda}_1- \widetilde{\Lambda}_2)U_1\overline{U}_2 dS \right |\leq \left \| \Lambda_1^{\sharp} - \Lambda_2^{\sharp} \right \| \left \| U_1 \right \|_{H^1(\Omega)}  \left \| U_2 \right \|_{H^1(\Omega)}.
\end{equation}
For the second term we will use the Carleman estimate from Proposition \ref{PCe}. Recall that we denoted by $N$ an open subset of $S^{n-1}$ as in the statement of Theorem \ref{SMP}. Hence, for every $\xi \in N$ and since $\chi$ is equal to $1$ on $\Omega_{-,\epsilon}(\xi)$, we get
\begin{equation}\label{secondtermt134}
\begin{aligned}
&\left |  \int_{\partial\Omega}(1-\chi)   (\widetilde{\Lambda}_1- \widetilde{\Lambda}_2)U_1\overline{U}_2 dS \right | \\
& = \left |  \int_{\Omega_{-,\epsilon}(\xi)\cup (\partial\Omega\setminus \Omega_{-,\epsilon}(\xi))} (1-\chi)(\widetilde{\Lambda}_1- \widetilde{\Lambda}_2)U_1\overline{U}_2 dS  \right | \\
& = \left |  \int_{\partial\Omega\setminus \Omega_{-,\epsilon}(\xi)}(1-\chi)(\widetilde{\Lambda}_1- \widetilde{\Lambda}_2)U_1\overline{U}_2 dS  \right | \\
&\leq  C_1 \left \| e^{-\tau \xi\cdot x}  (\widetilde{\Lambda}_1- \widetilde{\Lambda}_2)U_1 \right \|_{L^2(\partial\Omega\setminus\Omega_{-,\epsilon}(\xi))}  \left \| e^{\tau \xi\cdot x} U_2\right \|_{L^2(\partial\Omega\setminus \Omega_{-,\epsilon}(\xi))}.
\end{aligned}
\end{equation}
We now estimate the $L^2(\partial\Omega\setminus \Omega_{-,\epsilon}(\xi))$-norm in the above inequality. Let us introduce an auxiliary function $w_1$ satisfying
\begin{align}
\label{quq}
     \begin{cases}
            \mathcal{L}_{\widetilde{A}_2,q_2}w_1=0,  & 
            \\ {w_1}|_{\d \Omega} = U_1|_{\d \Omega} .
     \end{cases}
\end{align}
Now since $U_1\in H^1(\Omega)$ and  $\mathcal{L}_{\widetilde{A}_2,q_2}(w_1-U_1)= (\mathcal{L}_{\widetilde{A}_1,q_1} - \mathcal{L}_{\widetilde{A}_2,q_2}) U_1$, it follows that $\mathcal{L}_{A_2,q_2}(w-u_1)\in L^2(\Omega)$. Moreover, since $w_1$ satisfies (\ref{quq}), we have $w_1-U_1\in H^1_0(\Omega)$. Hence, the Carleman estimate (\ref{Ce}) and Remark \ref{Cer}, imply that
\begin{equation}\label{firststepq}
\begin{split}
& \left \| e^{-\tau \xi\cdot x}   (\widetilde{\Lambda}_1- \widetilde{\Lambda}_2)U_1 \right \|_{L^2(\partial\Omega\setminus \Omega_{-,\epsilon}(\xi))} \\
& =  \left\|    e^{-\tau \xi\cdot x}   \left( \partial_\nu(U_1-w_1) + i\nu\cdot (\widetilde{A}_1-\widetilde{A}_2)U_1   \right) \right\|_{L^2(\partial\Omega\setminus \Omega_{-,\epsilon}(\xi))}  \\
& \leq \left\|    e^{-\tau \xi\cdot x}    \partial_\nu(U_1-w_1) \right\|_{L^2(\partial\Omega\setminus \Omega_{-,\epsilon}(\xi))} \\
& \qquad   + \left \| \widetilde{A}_1-\widetilde{A}_2 \right \|_{L^\infty(\Omega)} \left\|    e^{-\tau \xi\cdot x}   U_1 \right\|_{L^2(\partial\Omega\setminus \Omega_{-,\epsilon}(\xi))} \\
& \leq  \dfrac{1}{\sqrt{\epsilon}}    \left\|   \sqrt{\left\langle \xi \cdot \nu(\cdot)  \right \rangle}    e^{-\tau \xi\cdot x}    \partial_\nu(U_1-w_1) \right\|_{L^2(\partial\Omega\setminus \Omega_{-,\epsilon}(\xi))}  \\
& \qquad   + \left \| \widetilde{A}_1-\widetilde{A}_2 \right \|_{L^\infty(\Omega)}  \left\|    e^{-\tau \xi\cdot x}   U_1 \right\|_{L^2(\partial\Omega\setminus \Omega_{-,\epsilon}(\xi))} \\
&  \leq  \dfrac{1}{\sqrt{\epsilon}}    \left\|   \sqrt{\left\langle \xi \cdot \nu(\cdot)  \right \rangle}    e^{-\tau \xi\cdot x}    \partial_\nu(U_1-w_1) \right\|_{L^2( \Omega_{+,0}(\xi))} \\
&\qquad +  \left \| \widetilde{A}_1-\widetilde{A}_2 \right \|_{L^\infty(\Omega)} \left\|    e^{-\tau \xi\cdot x}  U_1 \right\|_{L^2(\partial\Omega)}  \\
%&  \leq  \dfrac{C_2}{\sqrt{\epsilon}}\left( \left \|  \sqrt{-\left\langle \xi \cdot \nu(\cdot)  \right \rangle} e^{-\tau \xi\cdot x}  \partial_{\nu}(U_1-w_1) \right \|_{L^2(\partial\Omega_{-,0}(\xi))}  \right.\\
%&\qquad \qquad   \left. + \tau^{-\frac{1}{2}} \left \| e^{-\tau \xi\cdot x}   \mathcal{L}_{\widetilde{A}_2,q_2} (w_1-U_1) \right \|_{L^2(\Omega)}    \right) \\
%& \qquad \qquad \qquad  +  \left \| \widetilde{A}_1-\widetilde{A}_2 \right \|_{L^\infty(\Omega)} \left\|    e^{-\tau \xi\cdot x}  U_1 \right\|_{L^2(\partial\Omega)} \\
&   \leq  \dfrac{C_2}{\sqrt{\epsilon}}\left( \left \| e^{-\tau \xi\cdot x}  \partial_{\nu}(U_1-w_1) \right \|_{L^2(\partial\Omega_{-,0}(\xi))}  \right.\\
&\qquad \qquad   \left. + \tau^{-\frac{1}{2}} \left \| e^{-\tau \xi\cdot x}   (\mathcal{L}_{\widetilde{A}_1,q_1} - \mathcal{L}_{\widetilde{A}_2,q_2}) U_1 \right \|_{L^2(\Omega)}    \right) \\
& \qquad \qquad \qquad  +  \left \| \widetilde{A}_1-\widetilde{A}_2 \right \|_{L^\infty(\Omega)} \left\|    e^{-\tau \xi\cdot x}  U_1 \right\|_{L^2(\partial\Omega)}.
%& \qquad \qquad \qquad  \qquad \qquad  \qquad  \qquad \qquad \qquad + \left.  \left\|    e^{-\tau \xi\cdot x}   u_1 \right\|_{L^2(\partial\Omega)}  \right)
\end{split}
\end{equation}
Now we estimate the $L^2(\partial\Omega_{-,0}(\xi))$-norm in the last inequality as follows  
 \begin{equation}\label{secondstep}
\begin{aligned}
&  \left \|  e^{-\tau \xi\cdot x}  \partial_{\nu}(U_1-w_1) \right \|_{L^2(\partial\Omega_{-,0}(\xi))}  \\
& =   \left \| e^{-\tau \xi\cdot x}\left [ (\widetilde{\Lambda}_1-\widetilde{\Lambda}_2)U_1 -i \nu\cdot(\widetilde{A}_1-\widetilde{A}_2)U_1 \right ] \right \|_{L^2(\partial\Omega_{-,0}(\xi))}    \\
& \leq  \left \| e^{-\tau \xi\cdot x} (\widetilde{\Lambda}_1-\widetilde{\Lambda}_2)U_1  \right \|_{L^2(\partial\Omega_{-,0}(\xi))}   \\
& \qquad +  \left \| e^{-\tau \xi\cdot x} i \nu\cdot(\widetilde{A}_1-\widetilde{A}_2)U_1  \right \|_{L^2(\partial\Omega_{-,0}(\xi))}  \\
& =  \left \| e^{-\tau \xi\cdot x} \chi (\widetilde{\Lambda}_1-\widetilde{\Lambda}_2)U_1  \right \|_{L^2(\partial\Omega)}   \\
& \qquad +  \left \| e^{-\tau \xi\cdot x} i \nu\cdot(\widetilde{A}_1-\widetilde{A}_2)U_1  \right \|_{L^2(\partial\Omega_{-,0}(\xi))}  \\
& \leq  e^{\tau c}  \left \| \Lambda_1^{\sharp} - \Lambda_2^{\sharp} \right \| \left \| U_1 \right \|_{H^{\frac{1}{2}}(\partial \Omega)}  +    \left \| \widetilde{A}_1-\widetilde{A}_2 \right \|_{L^\infty(\Omega)} \left\|    e^{-\tau \xi\cdot x}  U_1 \right\|_{L^2(\partial\Omega)}. 
\end{aligned}
\end{equation}
Thus, replacing  (\ref{firststepq}) and (\ref{secondstep})  into  (\ref{secondtermt134}) gives us
\begin{equation} \label{unoqqq}
\begin{aligned}
& \left |  \int_{\partial\Omega}(1-\chi)   (\widetilde{\Lambda}_1- \widetilde{\Lambda}_2)U_1\overline{U}_2 dS \right | \\
& \leq C_4\left(   \epsilon^{-1/2}e^{\tau c}   \left \| \widetilde{\Lambda}_1^{\sharp} - \widetilde{\Lambda}_2^{\sharp} \right \| \left \| U_1 \right \|_{H^1(\Omega)}                                     \right.\\
& \qquad \qquad  +\epsilon^{-1/2} \tau^{-\frac{1}{2}} \left \| e^{-\tau \xi\cdot x}(\mathcal{L}_{\widetilde{A}_1,q_1} - \mathcal{L}_{\widetilde{A}_2,q_2}) U_1 \right \|_{L^2(\Omega)}  \\
& \qquad \qquad  \left.   +  \left \| \widetilde{A}_1-\widetilde{A}_2 \right \|_{L^\infty(\Omega)} \left\|    e^{-\tau \xi\cdot x}  U_1 \right\|_{L^2(\partial\Omega)} \right) \left \| e^{\tau \xi\cdot x} U_2\right \|_{L^2(\partial\Omega)}. 
\end{aligned}
\end{equation}
Replacing  (\ref{unoxx12}) and (\ref{unoqqq}) into (\ref{decomposition1}) we conclude the proof.

\end{proof}

\begin{cor}\label{estgi}
Let $\Omega\subset\mathbb{R}^n$ be a bounded open set with smooth boundary. Consider two positive constants $M$ and $\sigma$. Let $A_1, A_2\in \mathscr{A}(\Omega, M)$ with $A_1=A_2$ on $\partial\Omega$; and $q_1, q_2\in  \mathscr{Q}(\Omega, M, \sigma)$. If $U_1, U_2\in H^1(\Omega)$ satisfies $\mathcal{L}_{\widetilde{A}_1, q_1}U_1=0$ and $\mathcal{L}_{\widetilde{A}_2, \overline{q}_2}U_2=0$, then there exist three positive constants  $\tau_0, C$ and $\widetilde{\lambda}$ (all depending on $n, \Omega, M, \sigma$) such that  the estimate
%\begin{equation}\label{ggg5}
%\begin{aligned}
%& \left |  \left \langle (\widetilde{\Lambda}_1- \widetilde{\Lambda}_2)U_1, U_2  \right \rangle_{L^2(\partial\Omega)} \right |\\
%&  \leq C\left( e^{4\tau c} \left \| \Lambda_1^{\sharp}- \Lambda_2^{\sharp} \right \| +\tau^{1/2}  \left | \log \left | \log  \left \|\Lambda^{\sharp}_{1}- \Lambda^{\sharp}_{2}    \right \| \right |   \right |^{-\widetilde{\lambda}}  + \tau^{-1/2}\right)\left \|  \overline{g}\right \|_{H^2(\Omega)},
%\end{aligned}
%\end{equation}
\begin{equation}\label{ggg5}
\begin{aligned}
& \left |  \left \langle (\widetilde{\Lambda}_1- \widetilde{\Lambda}_2)U_1, U_2  \right \rangle_{L^2(\partial\Omega)} \right |\\
&  \leq C\left( e^{4\tau c} \left \| \Lambda_1^{\sharp}- \Lambda_2^{\sharp} \right \|  + \tau^{-1/2}\right)\left \|  \overline{g}\right \|_{H^2(\Omega)}\\
& \qquad +C\tau^{1/2}  \left | \log \left | \log  \left \|\Lambda^{\sharp}_{1}- \Lambda^{\sharp}_{2}    \right \| \right |   \right |^{-\widetilde{\lambda}} \left \|  \overline{g}\right \|_{H^2(\Omega)}
\end{aligned}
\end{equation}
holds true for all $\tau\geq \tau_0$.
\end{cor}

\begin{proof}
We start by considering $u_1, u_2\in H^{1}(\Omega)$, given by Theorem \ref{Zs} and \ref{DSFKSjUs}, respectively; satisfying $\mathcal{L}_{A_1, q_1}u_1=0$ and $\mathcal{L}_{A_2, \overline{q}_2}=0$. Thus, by identities (\ref{ginv1}) and (\ref{ginv2}) we have that $U_1=e^{i\omega/2}u_1$ and $U_2=e^{-i\omega/2}u_2$ satisfy
\[
\mathcal{L}_{\widetilde{A}_1, q_1}U_1=0, \quad  \mathcal{L}_{\widetilde{A}_2, \overline{q_2}}U_2=0.
\]
From (\ref{ggg2}), it follows that $U_1, U_2\in H^{1}(\Omega)$. Now take $p>n$. Since $A_1, A_2\in W^{2, \infty}(\Omega)$, we have that $A_1, A_2\in W^{2, p}(\Omega)$. The task now is to compute the norms corresponding to $U_1$ of  the right hand side of (\ref{alei}). The estimates for $U_2$ are similar. From (\ref{ceroestimate}) and  (\ref{ggg2}), we have
\begin{equation}\label{U100}
\begin{aligned}
 \left \| U_1 \right \|_{H^1(\Omega)} &=  \left \| e^{i\omega/2}u_1 \right \|_{H^1(\Omega)} = \left \| e^{i\omega/2}u_1 \right \|_{L^2(\Omega)} +   \left \| \nabla (e^{i\omega/2}u_1) \right \|_{L^2(\Omega)} \\
&= \left \| e^{i\omega/2}u_1 \right \|_{L^2(\Omega)} +   \left \| i(\nabla\omega/2) e^{i\omega/2} u_1+ e^{i\omega/2}\nabla u_1 \right \|_{L^2(\Omega)}\\
& \leq C_1 \left \| u_1 \right \|_{H^1(\Omega)}  \leq C_2 \tau e^{\tau c}.
\end{aligned}
\end{equation}
From (\ref{dosestimate}) and since $\omega=0$ on $\partial\Omega$ we obtain
\begin{equation}\label{U200}
\left \| e^{-\tau\xi\cdot x} U_1 \right \|_{L^2(\partial\Omega)} = \left \| e^{-\tau\xi\cdot x} e^{i\omega/2}u_1 \right \|_{L^2(\partial\Omega)} = \left \| e^{-\tau\xi\cdot x} u_1 \right \|_{L^2(\partial\Omega)} \leq C_3.
\end{equation}
To estimate the next term, we set $V=e^{i\omega/2} (a_1+r_1+e^{-\tau(\varphi+i\psi)}e^{\tau l}b)$, where $a_1,\varphi$ and $\psi$ as in (\ref{aaaaa5}). The functions $r_1, l$ and $b$ as in Theorem \ref{Zs}.  Thus, from (\ref{cuatroestimate}) we have
\begin{equation}\label{U300}
\begin{aligned}
&\left \|e^{-\tau\xi\cdot x} (\mathcal{L}_{\widetilde{A}_1, q_1} -   \mathcal{L}_{\widetilde{A}_2, \overline{q}_2} )  U_1\right \|_{L^2(\Omega)} \\
& = \left \| e^{-\tau( \varphi +i\psi)} (\mathcal{L}_{\widetilde{A}_1, q_1} -   \mathcal{L}_{\widetilde{A}_2, \overline{q}_2} )  \left [ e^{i\omega/2}  (e^{\tau(\varphi+i\psi )} (a_1+r_1)-e^{\tau l}b   ) \right ]   \right \|_{L^2(\Omega)}\\
& = \left \| e^{-\tau( \varphi +i\psi)} (\mathcal{L}_{\widetilde{A}_1, q_1} -   \mathcal{L}_{\widetilde{A}_2, \overline{q}_2} )  \left [   e^{\tau(\varphi+i\psi )} V\right ]   \right \|_{L^2(\Omega)}\\
& = \left \| 2\tau (\widetilde{A}_1- \widetilde{A}_2)\cdot D\rho V  + (\mathcal{L}_{\widetilde{A}_1, q_1} -   \mathcal{L}_{\widetilde{A}_2, \overline{q}_2})V\right \|_{L^2(\Omega)}\\
& \leq C_4 \left( \left \| \widetilde{A}_1-\widetilde{A}_2 \right \|_{L^\infty(\Omega)} \left \| V \right \|_{H^1(\Omega)} + \left \| V \right \|_{L^2(\Omega)} \right)\\
& \leq C_5 \left(  \tau \left \| \widetilde{A}_1-\widetilde{A}_2 \right \|_{L^\infty(\Omega)}  + 1\right).
\end{aligned}
\end{equation}
Analogously, from (\ref{unoestimate}) and (\ref{tresestimate}) we obtain
\begin{equation}\label{U400}
\left \| U_2 \right \|_{H^1(\Omega)}    \leq C_6 \tau e^{\tau c} \left \| \overline{g} \right \|_{H^2(\Omega)}, \quad \left \| e^{\tau\xi\cdot x} U_2 \right \|_{L^2(\partial\Omega)}   \leq C_7 \left \| \overline{g} \right \|_{H^2(\Omega)}
\end{equation}
Thus, taking into account that there exists $C_8>0$ such that $\tau\leq C_8e^{\tau k}$ for $\tau$ large enough and combining the estimates (\ref{U100})-(\ref{U400}), we obtain
\begin{equation}\label{ggg4}
\begin{aligned}
& \left |  \left \langle (\widetilde{\Lambda}_1- \widetilde{\Lambda}_2)U_1, U_2  \right \rangle_{L^2(\partial\Omega)} \right |\\
&  \leq C_9\left( e^{4\tau c} \left \| \widetilde{\Lambda}_1^{\sharp}-  \widetilde{\Lambda}_2^{\sharp} \right \| +\tau^{1/2} \left \| \widetilde{A}_1-\widetilde{A}_2 \right \|_{L^\infty(\Omega)}  + \tau^{-1/2}\right)\left \|  \overline{g}\right \|_{H^2(\Omega)}.
\end{aligned}
\end{equation}
On the other hand, we fix $q\in \mathbb{R}$ such that $n<p<q$, and consider $t\in \left( 0,1 \right)$ satisfying $1/p=t/2+(1-t)/q$. Then by elementary interpolation we have
\[
\left \| d(A_1-A_2) \right \|_{L^p(\Omega)}\leq \left \| d(A_1-A_2) \right \|_{L^2(\Omega)}^{t} \left \| d(A_1-A_2) \right \|_{L^q(\Omega)}^{1-t}.
\]
Hence, Theorem \ref{SMP} and (\ref{ggg1}), imply that
\begin{equation}\label{hhh5}
\left \| A_1-A_2-\nabla \omega \right \|_{C^{0, 1-\frac{n}{p}}(\Omega)} \leq C_{10} \left | \log \left | \log  \left \|\Lambda^{\sharp}_{A_1,q_1}- \Lambda^{\sharp}_{A_2,q_2}    \right \| \right |   \right |^{-t\lambda/2}.
\end{equation}
Observe that from (\ref{ginv1})-(\ref{ginv2}), we have $\left \| \widetilde{\Lambda}_1^{\sharp}-  \widetilde{\Lambda}_2^{\sharp} \right \| = \left \|\Lambda_1^{\sharp}-  \Lambda_2^{\sharp} \right \| $. Moreover, $\widetilde{A}_1- \widetilde{A}_2= A_1-A_2-\nabla \omega$. So we conclude the proof by combining the above inequality and (\ref{ggg4}).
\end{proof}

Corollary \ref{estgi} gives us an estimate for the left hand side of Alessandrini's identity (\ref{algee}). The task now is to isolate $q_1-q_2$ from the right hand side. Thus, from (\ref{algee}) we have
\begin{equation}\label{hhh}
\begin{aligned}
&\left |  \int_{\Omega} (q_1-q_2)U_1\overline{U}_2  \right |\leq \left | \left \langle  (\widetilde{\Lambda}_1 - \widetilde{\Lambda}_2)U_1,U_2 \right \rangle_{L^2(\partial\Omega)}  \right |\\
&  \qquad  \qquad +\left | \int_{\Omega}  (\widetilde{A}_1-\widetilde{A}_2)\cdot(DU_1 \overline{U}_2 + U_1 \overline{DU}_2) \right | \\
& \qquad \qquad  \qquad+\left |  \int_{\Omega} (\widetilde{A}_1-\widetilde{A}_2)\cdot  (\widetilde{A}_1+\widetilde{A}_2)U_1\overline{U}_2  \right |\\
& \leq \left | \left \langle  (\widetilde{\Lambda}_1 - \widetilde{\Lambda}_2)U_1,U_2 \right \rangle_{L^2(\partial\Omega)}  \right | \\
&\;\; \; \;  +C_1 \left \|  \widetilde{A}_1-\widetilde{A}_2\right \|_{L^{\infty}} \left(  \left \| DU_1 \overline{U}_2 + U_1 \overline{DU}_2 \right \|_{L^1(\Omega)} +\left \| U_1 \overline{U}_2 \right \|_{L^1(\Omega)} \right).
 \end{aligned}
\end{equation}

Recall that $U_1=e^{i\omega/2}u_1$ and $U_2=e^{-i\omega/2}u_2$, where $u_1, u_2\in H^1(\Omega)$ satisfy $\mathcal{L}_{A_1,q_1}u_1=0$ and $\mathcal{L}_{A_2, \overline{q}_2}u_2=0$, respectively. Hence, from  (\ref{ggg2}),  (\ref{a4})-(\ref{a5}) and an easy computation we have that  
\begin{equation}\label{hhh1}
\begin{aligned}
&  \left \| DU_1 \overline{U}_2 + U_1 \overline{DU}_2 \right \|_{L^1(\Omega)} +\left \| U_1 \overline{U}_2 \right \|_{L^1(\Omega)}\\
& \qquad  \leq C_2\left(  \left \| Du_1 \overline{u}_2 + u_1 \overline{Du}_2 \right \|_{L^1(\Omega)} +\left \| u_1 \overline{u}_2 \right \|_{L^1(\Omega)}  \right)  \leq C_3 \tau.
\end{aligned}
\end{equation}
We consider now $u_1\in H^1(\Omega)$ as in Theorem \ref{Zs}, $u_2\in H^1(\Omega)$ as in Theorem \ref{DSFKSjUs}. As in (\ref{aaaaa5}) and (\ref{denoted}) we denote $a_1=e^{\Phi_1}$ and $a_2=e^{\Phi_2}g$, where $g$ is any smooth function satisfying (\ref{guno1}).Thus, we have
\begin{align*}
& \int_{\Omega}e^{i\omega}(q_1-q_2) a_1\overline{a}_2= \int_{\Omega} (q_1-q_2)U_1\overline{U}_2\\
&\qquad \qquad  - \int_{\Omega} e^{i\omega}(q_1-q_2) \left [ a_1\overline{r}_2 +r_1\overline{a}_2+r_1\overline{r}_2 + e^{-\tau(\varphi+i\psi)}e^{\tau l} b (\overline{a}_2+\overline{r}_2)   \right ],
\end{align*}
and combining (\ref{hhh})-(\ref{hhh1}) with (\ref{estar1}), (\ref{estar12}) and (\ref{estar123}); we obtain
\begin{align*}
\left |  \int_{\Omega}e^{i\omega}(q_1-q_2) a_1\overline{a}_2 \right |&  \leq   \left | \left \langle  (\widetilde{\Lambda}_1 - \widetilde{\Lambda}_2)U_1,U_2 \right \rangle_{L^2(\partial\Omega)}  \right |  \\
& \qquad +C_4\tau \left \|  \widetilde{A}_1-\widetilde{A}_2\right \|_{L^{\infty}} + C_5\tau^{-1}.
\end{align*}
This inequality, (\ref{hhh5}) and Corollary \ref{estgi}, imply that there exist two positive constants $\tau_0$ and $C_6$ such that
\begin{equation}\label{ggg6}
\begin{aligned}
& \left |  \int_{\Omega}e^{i\omega}(q_1-q_2) a_1\overline{a}_2\right |  \leq C_6 \left \|  \overline{g}\right \|_{H^2(\Omega)} \\
&\qquad \qquad  \times \left( e^{4\tau c} \left \| \Lambda_1^{\sharp}-  \Lambda_2^{\sharp} \right \| +\tau \left | \log \left | \log  \left \|\Lambda^{\sharp}_{1}- \Lambda^{\sharp}_{2}    \right \| \right |   \right |^{-\widetilde{\lambda}}  + \tau^{-1/2}\right),
\end{aligned}
\end{equation}
for all $\tau\geq \tau_0$.
\begin{prop}\label{removepe}
Let $\Omega\subset\mathbb{R}^n$ be a bounded open set with smooth boundary. Consider two positive constants $M$ and $\sigma$. Let $A_1, A_2\in \mathscr{A}(\Omega, M)$ with $A_1=A_2$ on $\partial\Omega$; and $q_1, q_2\in  \mathscr{Q}(\Omega, M, \sigma)$. Consider any smooth function $g$ satisfying $(\xi +i\zeta)\cdot \nabla g =0$ (see (\ref{guno1})). If $A_1=A_2$ on $\partial\Omega$, then there exist two positive constants $\tau_0$ and $C$ (both depending on $n, \Omega, M.\sigma$) such that 
\begin{equation}\label{Removeq}
 \left | \int_{\Omega}(q_1-q_2) \overline{g}\right |\leq C\left | \log \left | \log  \left \|\Lambda^{\sharp}_{1}- \Lambda^{\sharp}_{2}    \right \| \right |   \right |^{-\frac{\widetilde{\lambda}}{3}}  \left \| \overline{g} \right \|_{H^2(\Omega)},
\end{equation}
provided that  $\left \| \Lambda_1^{\sharp}-  \Lambda_2^{\sharp} \right \| \leq  e^{-e^{ \left( 8c\tau_0\right)^{\frac{3}{2}\widetilde{\lambda}^{-1}}}}$.
\end{prop}

\begin{proof}
We start with the following identity
\begin{equation}\label{lll1}
\int_{\Omega}(q_1-q_2) \overline{g}= \int_{\Omega}( 1-e^{\Phi_1+\overline{\Phi}_2+i\omega} )(q_1-q_2)\overline{g} +   \int_{\Omega}e^{i\omega}(q_1-q_2) a_1\overline{a}_2,
\end{equation}
From (\ref{Phiuno}) and (\ref{Phidos}), we have
\[
(\xi +i\zeta)\cdot \nabla (\Phi_1+\overline{\Phi}_2)+i (\xi +i\zeta)\cdot (A_1-A_2)=0,
\]
which imply that
\[
(\xi +i\zeta)\cdot \nabla (\Phi_1+\overline{\Phi}_2 +i\omega)+i (\xi +i\zeta)\cdot (A_1-A_2-\nabla\omega)=0
\]
and  by estimate (\ref{infinito}) from Remark \ref{infi}, we get
\[
\left \| \Phi_1+\overline{\Phi}_2 +i\omega \right \|_{L^\infty(\Omega)}\leq C_1 \left \| A_1-A_2-\nabla\omega \right \|_{L^\infty(\Omega)}.
\]
We can now estimate the first term of the right hand side of (\ref{lll1}) by using the inequality
\[
\left | e^{a}-e^{b} \right |\leq \left | a-b \right | e^{\max \left \{ \Re a, \Re b \right \}} \; \; , \; \; a,b \in \mathbb{C}.
\]
Thus,
\begin{align*}
& \left | \int_{\Omega}( 1-e^{\Phi_1+\overline{\Phi}_2+i\omega} )(q_1-q_2)\overline{g}   \right |= \left | \int_{\Omega}( e^0-e^{\Phi_1+\overline{\Phi}_2+i\omega} )(q_1-q_2)\overline{g}   \right |\\
& \qquad\leq \left \| ( \Phi_1+\overline{\Phi}_2 +i\omega) e^{\max\left \{ 0, \Re( \Phi_1+\overline{\Phi}_2 +i\omega) \right \}} \right \|_{L^\infty(\Omega)}\int_{\Omega} \left | (q_1-q_2) \overline{g} \right |\\
& \qquad\leq C_2  \left \| A_1-A_2-\nabla\omega \right \|_{L^\infty(\Omega)} \left \|  \overline{g}\right \|_{L^2(\Omega)}.
\end{align*}
Taking into account  (\ref{lll1}), (\ref{hhh5}) and (\ref{ggg6}) we know that there exist $\tau_0>0$ and  $C_3>0$ such that 
\begin{equation}
\begin{aligned}
& \left | \int_{\Omega}(q_1-q_2) \overline{g}\right | \leq C_3\left \|  \overline{g}\right \|_{H^2(\Omega)} \\
&  \times \left( e^{4\tau c} \left \| \Lambda_1^{\sharp}-  \Lambda_2^{\sharp} \right \| +\tau \left | \log \left | \log  \left \|\Lambda^{\sharp}_{1}- \Lambda^{\sharp}_{2}    \right \| \right |   \right |^{-\widetilde{\lambda}}  + \tau^{-1/2}\right).
\end{aligned}
\end{equation}
We conclude the proof by taking
\[
\tau= \dfrac{1}{8c} \left | \log \left | \log  \left \|\Lambda^{\sharp}_{1}- \Lambda^{\sharp}_{2}    \right \| \right |   \right |^{\frac{2}{3}\widetilde{\lambda}} \geq \tau_0,
\]
whenever
\[
\left \| \Lambda_1^{\sharp}-  \Lambda_2^{\sharp} \right \| \leq  e^{-e^{ \left( 8c\tau_0\right)^{\frac{3}{2}\widetilde{\lambda}^{-1}}}}.
\]
\end{proof}
\subsection{Proof of Theorem \ref{SEP}} We begin by considering the notation introduced in Theorem \ref{SMP} and proceed analogously as in its proof. The estimate (\ref{Removeq}) from Proposition \ref{removepe},  imply that
\[
\left | \int_{\mathbb{R}}  \widetilde{g}(s) (\pmb{R}\left [  \chi_{\Omega}(q_1-q_2)\right ] )(s, \theta )ds \right |\leq C_1 \left | \log \left | \log  \left \|\Lambda^{\sharp}_{1}- \Lambda^{\sharp}_{2}    \right \| \right |   \right |^{-\frac{\widetilde{\lambda}}{3}}\left \| \widetilde{g} \right \|_{H^2(\mathbb{R})},
\]
for all $\theta\in M$. The set $M$ is defined by (\ref{M}). From this inequality we have
\[
\left \| \pmb{R}\left(\chi_{\Omega}(q_1-q_2)\right)  \right \|_{H^{-2}(\mathbb{R}; L^{\infty}(M))}  \leq  C_2 \left | \log \left | \log  \left \|\Lambda^{\sharp}_{1}- \Lambda^{\sharp}_{2}    \right \| \right |   \right |^{-\frac{\widetilde{\lambda}}{3}}.
\]
On the other hand, from (\ref{Rfequation}), we get
\[
 \left \| \pmb{R}\left(\chi_{\Omega}(q_1-q_2) \right)  \right \|_{H^{\frac{n-1}{2}}(\mathbb{R}; L^2(M))} \leq C_3  \left \| \chi_{\Omega}(q_1-q_2)   \right \|_{L^{2}(\mathbb{R}^n)}\leq C_4.
\]
Thus by standard interpolation between the spaces $H^{-2}(\mathbb{R}; L^{\infty}(M))$ and $H^{\frac{n-1}{2}}(\mathbb{R}; L^2(M))$, we obtain
\begin{equation}\label{interpolaaaaa}
\begin{aligned}
&  \left \| \pmb{R}\left(\chi_{\Omega}(q_1-q_2) \right)  \right \|_{L^2(\mathbb{R}; L^{{(n+3)}/{2}}(M))}\\
 & \qquad \qquad \qquad \leq C_3 \left | \log \left | \log  \left \|\Lambda^{\sharp}_{1}- \Lambda^{\sharp}_{2}    \right \| \right |   \right |^{-\frac{\widetilde{\lambda}}{3}{(n-1)}/{(n+3)}}.
\end{aligned}
\end{equation}

We are now in position to apply Theorem \ref{cdsr} to the function $\chi_{\Omega}(q_1-q_2)$. Let us verify its three conditions. Since $\Omega$ is bounded, the supporting condition $(b)$ is satisfied for some $y_0\in \mathbb{R}^n$. From the above estimate, there exists $\beta\in \left( 0,1 \right)$ such that the condition $(a)$ is satisfied for any $\alpha>0$. Thus, by taking $\alpha>0$ large enough it follows that $\supp\left( \chi_\Omega(q_1-q_2)\right) \subset G$. Since  $q_1,q_2\in H^{\sigma}(\mathbb{R}^n)$ and $\chi_{\Omega}\in H^{1/2-\sigma}(\mathbb{R}^n)$ (for this last fact see \cite{FRo}), the condition $(c)$ is satisfied for $p=2$ and $0<\lambda<1/2$. For convenience we set $q=\chi_{\Omega}(q_1-q_2)$. Then Theorem \ref{cdsr} ensures that there exists $C_4>0$ such that
\begin{equation}\label{a12345}
\left \| q\right \|_{L^2(\mathbb{R}^n)} \leq C_4    \left |  \log \int_{-\alpha}^{\alpha}  (1+ \left | s \right |)^n \left \| \pmb{R}_{y_0} q(s, \cdot) \right \|_{L^1(\Gamma)} ds   \right |^{-\lambda/2}.
\end{equation}
Analogously to the proof of the magnetic potentials, here the set $\Gamma$ is where we have the control of the Radon transform on the $\theta$-variable. In our case (see the estimate (\ref{interpolaaaaa})) we have the control on $M$. Now we set
\[
L= \underset{\theta\in M}{\sup}  \left \| (1+ \left | \cdot- \left \langle \theta, y_0 \right  \rangle \right |)^n  \right \|_{L^2(\left | s \right |\leq \alpha +\left | y_0 \right | )}
\]
and denote by $\left | M \right |$ the measure of $M$. Then the inequality (\ref{interpolaaaaa}), Fubini's theorem and H\"older's inequality applied twice, and a repetition of the arguments at the end of the proof of Theorem \ref{SMP} will give us

\begin{align*}
&\int_{-\alpha}^{\alpha}  (1+ \left | s \right |)^n \left \| \pmb{R}_{y_0} q(s, \cdot) \right \|_{L^1(M)} ds\\
%&=  \int_{-\alpha}^{\alpha}  (1+ \left | s \right |)^n  \int_M  \left | (\pmb{R} q )(s+ \left \langle \theta, y_0 \right  \rangle, \theta)   \right | d\theta ds \\
%& \leq\int_M \int_{-(\alpha +\left | y_0 \right | )}^{\alpha +\left | y_0 \right | }    (1+ \left | s- \left \langle \theta, y_0 \right  \rangle \right |)^n  \left | (\pmb{R} q )(s, \theta)   \right | ds  d\theta  \\
%& \leq \int_M   \left \| (1+ \left | \cdot- \left \langle \theta, y_0 \right  \rangle \right |)^n  \right \|_{L^2(\left | s \right |\leq \alpha +\left | y_0 \right | )}    \left \| (\pmb{R} q )(\cdot, \theta)  \right \|_{L^2(\left | s \right |\leq \alpha +\left | y_0 \right | )} d\theta \\
%&\leq  L \int_M \left(  \int_{\mathbb{R}}  \left | \pmb{R}q (s,\theta) \right | ^2 ds \right)^{1/2}d\theta  \\
%&  \leq L \left | M \right |^{\frac{n+1}{n+3}}  \left(   \int_M   \left(  \int_{\mathbb{R}} \left | \pmb{R}q  (s, \theta)  \right |^2 ds      \right)^{(n+3)/4}   d\theta \right)^{2/(n+3)}\\
%& =  L \left | M \right |^{\frac{n+1}{n+3}}  \left \| \pmb{R} q\right \|_{L^2(\mathbb{R}; L^{{(n+3)}/{2}}(M))}\\
&  \leq C_5 \left | \log \left | \log  \left \|\Lambda^{\sharp}_{1}- \Lambda^{\sharp}_{2}    \right \| \right |   \right |^{-\frac{\widetilde{\lambda}}{3}{(n-1)}/{(n+3)}}.
\end{align*}
We conclude the proof by taking logarithms in both sides of the above inequality and taking into account the estimate (\ref{a12345}).

\section*{Acknowledgements} The author would like to thank Alberto Ruiz  for the very nice  discussions about Mathematics. For his support and encouragement during the preparation of this article. I would also want to thank Pedro Caro for fruitful conversations about Mathematics and several comments about this article. I would also want to thank Mikko Salo for his hospitality during my research stay in Jyv\"askyl\"a and also for several nice conversations about Mathematics. This article is part of my PhD dissertation and it is supporting by the Project MTM$2011-28198$ of Ministerio de Econom\'ia y Competividad de Espa\~na.

\end{document}